\documentclass[12pt,a4paper]{amsart}
\usepackage[utf8]{inputenc}
\usepackage{amsmath}
\usepackage{amsfonts}
\usepackage{amssymb}
\usepackage{amsthm}
\usepackage{mathrsfs}
\usepackage{color}
\usepackage{a4wide}
\usepackage{hyperref}

\newcommand{\el}{\mbox{$\mathcal{L}$}}

\newcommand{\ar}{\mbox{$\mathcal R$}}

\newcommand{\els}{\mbox{${\mathcal L}^{\ast}$}}

\newcommand{\ars}{\mbox{${\mathcal R}^{\ast}$}}
\newcommand{\art}{\mbox{$\widetilde{\mathcal{R}}$}}
\newcommand{\elt}{\mbox{$\widetilde{\mathcal{L}}$}}

\theoremstyle{plain}
\newtheorem{Thm}{Theorem}[section]
\newtheorem{Cor}[Thm]{Corollary}
\newtheorem{Lem}[Thm]{Lemma}
\newtheorem{Prop}[Thm]{Proposition}

\newtheorem{lemma}[Thm]{Lemma}

\theoremstyle{definition}		% makes these non-italic
\newtheorem{defn}[Thm]{Definition}
\newtheorem{Rem}[Thm]{Remark}

\newtheorem{Question}[Thm]{Question}

\begin{document}
\title{On graph products of monoids}

\begin{abstract} Graph products of monoids provide a common framework for direct and free products, and graph monoids (also known as free partially commutative monoids). If the monoids in question are groups then any graph product is a group. For monoids that are not groups, regularity is perhaps the first and most important algebraic property that one considers:  however, graph products of regular monoids are not in general regular. We show that  a graph product of regular monoids satisfies  the related, but weaker, condition of  being abundant. More generally, we show that the classes of left abundant and left Fountain monoids are closed under graph product.
As a very special case we obtain the earlier result of Fountain and Kambites that the graph product of right cancellative monoids is right cancellative. To achieve our aims we show that elements in (arbitrary) graph products have a unique Foata normal form, and give some useful reduction results; these may equally well be applied to groups as to the broader case of monoids.
\end{abstract}

\date{\today}



\thanks{The research was supported by Grant No.\ 12171380 of
the National Natural Science Foundation of China, and by Grant No. QTZX2182 of the Fundamental
Research Funds for the Central Universities.}

\author[Y. Dandan]{Yang Dandan}
\address{School of Mathematics and Statistics, Xidian University, Xi'an 710071, P. R. China }
\email{ddyang@xidian.edu.cn}

\author[V. Gould]{Victoria Gould}
\address{Department of Mathematics, University of York, Heslington, York, YO10 5DD, UK}
\email{victoria.gould@york.ac.uk}

\maketitle

\section{Introduction} \label{sec:intro}

Graph products  arise from many sources and provide an important and wide ranging construction. They are defined by presentations, where the edges of a simple, non-directed graph determine commutativity of  elements   associated with the vertices. Further details are given in Section~\ref{sec:preliminaries}. {\em Graph products of  monoids} are defined in the same way as {\em graph products of  groups}, a notion introduced by Green in her thesis \cite{Green:1990},  and generalise at one and the same time  free products, restricted direct products, free (commutative)
monoids  and graph monoids\footnote{The existing terminology is a little unfortunate. Graph monoids are a strict subclass of the class of graph products of monoids. Note also that graph groups should not be confused with  the fundamental groups of graphs of groups.}. The latter are graph products of free monogenic monoids,
and were introduced by Cartier and Foata \cite{cartier:1969} to study combinatorial problems for rearrangements of words; they have  been extensively studied by mathematicians and computer scientists, having  applications  to
the study of concurrent processes \cite{diekert:1990,diekert:1995}. Graph monoids are also known as
free partially commutative monoids, right-angle Artin monoids  and trace monoids (sometimes with the condition the underlying graph is finite);  corresponding terminology applies in the case for groups.  Graph  groups were first defined by Baudisch \cite{baudisch:1981}; for a recent survey  see \cite{duncan:2020} and for the analogous notion for  {\em inverse semigroups} see \cite{daCosta:2003, diekert:2008}.

 Although mentioned in \cite{Green:1990} and in other earlier works focussing on groups, graph products of monoids per se were first defined in \cite{dacosta:2001}, and have subsequently been studied in various contexts, e.g. \cite{dacosta:2001, fohry:2006}.
 Much of the existing work in graph products of monoids, and groups, has been to show that various properties  are preserved under graph product, see e.g. \cite{hermiller:1995, diekert:2008(1), dacosta:2001(2), karpuz:2016}. These  properties are often  of algorithmic type, for example, automaticity \cite{hermiller:1995, dacosta:2001(2)}. In a different direction,  articles such as \cite{antolin:2015,barkauskas:2007,fountain:2009} consider  algebraic conditions. Of particular interest to us here is that Fountain and Kambites \cite{fountain:2009} show that a graph product of right cancellative monoids is right cancellative.

A monoid $M$ is {\em regular} if for any $a\in M$ there is a  $b\in M$ such that $a=aba$; note that   $ab,ba$ are, respectively, idempotent left and right identities for $a$. From an algebraic point of view, regularity is often the first property to look for in a monoid. Yet, it is easy to see that only in very special cases will a graph product of regular monoids be regular.

 The aim of this paper is easy to state. We consider two properties that each provide a natural  weakening of  regularity, and show that the classes of monoids satisfying
 these properties are closed under graph product. In general, the properties we consider provide the natural framework to study classes of monoids that need not be regular, but which have behaviour strongly influenced by idempotent elements.   We first prove:

 \bigskip
 {\bf Theorem~\ref{thm:mainm}} {\em The graph product of left abundant monoids is left abundant}.

 \bigskip
  A monoid $M$ is {\em left abundant} if every principal left ideal is projective  (so that sometimes a left abundant monoid is called left PP \cite{fountain:1977}).  This property may handily be expressed by saying that every $\ars$-class  of $M$ contains an idempotent. We define the relation $\ars$  in Section~\ref{sec:preliminaries}; it suffices to say here that $\ars$ contains Green's relation $\ar$, whence it follows immediately that regular semigroups are left abundant.
  We note that a monoid is a single $\ars$-class if and only if it is right cancellative.  Certainly then such monoids are abundant. The above mentioned result of \cite{fountain:2009} easily follows.

 \bigskip

   {\bf Corollary~\ref{cor:canc}} \cite[Theorem 1.5]{fountain:2009} {\em
The graph product  of right cancellative monoids  is  right cancellative.}

 \bigskip

Our second main result is:

 \bigskip
 {\bf Theorem~\ref{thm:fountain}} {\em The graph product of left Fountain monoids is left Fountain}.

 \bigskip
 
One way to define a {\em left Fountain}  (also known as {\em weakly left abundant}, or {\em left semiabundant}) monoid $M$ is to say that every $\art$-class  of $M$ must contain an idempotent; we give further details in
 Section~\ref{sec:preliminaries}.  Here  $\art$ is a relation containing $\ars$, whence it is clear that left abundant monoids are left Fountain. As for left abundancy, there is a natural approach to left Fountainicity using principal one-sided ideals. Again  as for left abundancy, such semigroups
 arise independently from a number of sources. They (and their two-sided versions) appear in the
work
 of de Barros
\cite{deB}, in that of Ehresmann on certain small ordered categories
\cite{RE}
and in the thesis of El Qallali \cite{E}. A systematic study of
such semigroups was initiated by Lawson, who establishes in \cite{lawson:1991}
the connection with Ehresmann's work. A useful source for the genesis of these ideas is Holling's survey \cite{hollings:2009}. We note here that the class of left Fountain monoids contains a number of important subclasses: we have mentioned left abundant, but we also  have left ample and left restriction \cite{hollings:2009}.  The study of left abundant monoids,  left Fountain monoids, their two-sided versions, and monoids in related classes, continues to provide one focus in algebraic semigroup theory. Some results show similarities  with
the structure of regular and inverse monoids \cite{guo:2005,fountaingomesgould:2009}, whereas others illustrate significantly different  behaviour \cite{kambites:2011,wang:2016,branco:2018}.

In order to prove Theorems~\ref{thm:mainm} and \ref{thm:fountain}  we have considerable work to do to get a grip on normal forms of elements of graph products. Essentially, the difficulty in the transition from graph monoids to graph products of monoids lies in the fact that for the broader concept the group of units of the monoids in question need not be trivial. Some of our techniques and results concerning normal forms and reduction of products of words may be of independent interest. In particular, in Proposition~\ref{lem:leftreducedform}, we establish that elements in graph products of monoids have a left Foata normal; previously this was an important tool in the study of graph monoids, and the same holds here.

The structure of this paper is as follows. In Section~\ref{sec:preliminaries} we give the necessary definitions and gather together the results we need from the literature. In Section~\ref{sec:lfnf} we begin our analysis of the form of words, and how these behave with respect to products. We establish the left Foata normal form for elements of graph products, not relying on any assumption of cancellativity.  In the next two sections we build a suite of techniques that allow us to simplify the words we need to consider when determining the relation $\ars$, these then enable us eventually to prove Theorem~\ref{thm:mainm}.
In Section~\ref{sec:Fountainicity} we use the earlier techniques, together with a further  analysis of words, to establish Theorem~\ref{thm:fountain}.  There is a corresponding notion of graph product for {\em semigroups}; the behaviour of the resulting semigroup  is similar to that of  a graph monoid and hence sheds some of the technical difficulties we encounter in graph products of monoids.  We apply our results to the semigroup case in  Section~\ref{sec:application},  and mention  a number of other applications. We finish with some open questions.

\section{Preliminaries}\label{sec:preliminaries}

We outline the notions required to read this article. For further details, we recommend the classic texts \cite{cliffordpreston} and \cite{howie:1995}.

\subsection{Presentations and graph products of monoids} We begin with an account of the notion on which this article is based: that of graph product of monoids. They are determined by monoid presentations. Let $X$ be a set. The  {\em free monoid} $X^*$ on $X$ consists of all words over $X$ with operation of juxtaposition.
We denote a  non-empty word  by $x_1\circ \cdots \circ x_n$ where $x_i\in X$ for $1\leq i\leq n$; we also
 use $\circ$ for juxtaposition of words. The empty word is denoted by $\epsilon$ and is the identity of $X^*$. Throughout, our convention is that if we say $x_1\circ\cdots \circ x_n\in X^*$, then we mean that $x_i\in X$ for all $1\leq i\leq n$, unless we explicitly say otherwise. We write $|x|$ for the length of a word $x=x_1\circ\cdots \circ x_n\in X^*$  and  denote by $x^r$  the word $x_n\circ \cdots \circ x_1\in X^*$.

A {\em monoid presentation} $\langle X\mid R\rangle$, where $X$ is a set and $R\subseteq  X^*\times X^*$, determines the monoid $X^*/R^{\sharp}$, where $R^{\sharp}$ is the congruence on $X^*$ generated by $R$. In the usual way, we identify $(u,v)\in R$ with the formal equality $u=v$ in a presentation $\langle X\mid R\rangle$.

We now define  graph products of monoids  \cite{Green:1990,dacosta:2001}.  Let $\Gamma=\Gamma(V,E)$ be a  simple,  undirected,  graph with no loops.
Here $V$ is a  non-empty set of {\em vertices} and  $E\subseteq V_2$ is  the set of {\em edges} of  $ \Gamma$, where $V_2$ is the set of $2$-element subsets of $V$.
We think of $\{ \alpha,\beta\}\in E$ as joining the vertices  $\alpha,\beta\in V$.
For notational reasons we denote an edge $\{ \alpha,\beta\}$ as $(\alpha,\beta)$ or $(\beta,\alpha)$; since our graph is undirected we are  identifying $(\alpha,\beta)$ with $(\beta,\alpha)$.

 \begin{defn}\label{defn:graphprodmonoids}  Let
$\Gamma=\Gamma(V,E)$ be a graph and let $\mathcal{M}=\{M_\alpha: \alpha\in V\}$ be a set of mutually disjoint monoids. We write $1_{\alpha}$ for the identity of $M_{\alpha}$ and put $I=\{ 1_{\alpha}: \alpha\in V\}$.   The {\em graph product} $\mathscr{GP}=\mathscr{GP}(\Gamma,\mathcal{M})$ of
 $\mathcal{M}$   with respect to $\Gamma$  is
the monoid defined by the presentation

 \[\mathscr{GP}=\langle X\mid R  \rangle\]
 where $X=\bigcup_{\alpha\in V}M_\alpha$ and $R= R_{id}\cup R_v\cup R_e$ are given by:

\[\begin{array}{rcl}R_{id}&=&\{   1_\alpha=\epsilon:  \alpha\in V\},\\

R_v&=&\{   x\circ y=xy:  \ x,y\in M_\alpha,\alpha\in V\},\\

R_e&=& \{ x\circ y=y\circ x: x\in M_\alpha,\, y\in M_\beta, (\alpha,\beta)\in E)\}.\end{array}\]
\end{defn}

The monoids $M_{\alpha}$ in Definition~\ref{defn:graphprodmonoids} are known as {\em vertex monoids}.
Throughout we assume $|V|\geq 2$, as otherwise $\mathscr{GP}$ is isomorphic to the single vertex monoid. We denote the
$R^{\sharp}$-class of $w\in X^*$ in $\mathscr{GP}$ by $[w]$. It is worth noting that there are various different ways to set up graph products, which all yield equivalent constructions. In particular, if one starts with monoids that are groups, the process above yields the graph product of groups.

The main focus of this article is on  monoids, although we briefy visit graph products of {\em semigroups} in Section~\ref{sec:application}. Free  products of semigroups, and a discussion of their universal properties, may be found in \cite{cliffordpreston,howie:1995}. Free products of monoids
may be viewed as a special case of an amalgamated free product of semigroups; this is commented on explicitly in \cite[p. 266]{howie:1995}. Here we remark that a free product of monoids is a graph product  for a graph $\Gamma(V,\emptyset)$.

We now touch on the other extreme where $E=V_2$.
Let $\mathcal{M}=\{M_\alpha: \alpha\in V\}$  be as above. The {\em restricted direct product} (or {\em direct sum})  $\oplus_{\alpha\in V}M_\alpha$ of $\mathcal{M}$ is defined by
\[\oplus_{\alpha\in V}M_\alpha=\{ f\in\Pi_{\alpha\in V}M_\alpha: \alpha f\neq 1_v\mbox{ for only finitely many }v\in V\}.\]
Clearly $\oplus_{\alpha\in V}M_\alpha$ is a submonoid of $\Pi_{\alpha\in V}M_\alpha$ and
${\oplus_{\alpha\in V}M_\alpha}=\Pi_{\alpha\in V}M_\alpha$ if and only if $V$ is finite. It is easy to see that a restricted direct product of monoids is a graph product for a graph
$\Gamma(V,V_2)$.

Graph products of monoids behave beautifully with respect to certain substructures, as we now demonstrate. To do so
we need some terminology.

\begin{defn}\label{defn:support} Let $\mathscr{GP}=\mathscr{GP}(\Gamma,\mathcal{M})$.
Let $s: X\rightarrow V$ be a map defined by $s(a)=\alpha$ if $a\in M_\alpha$.
The {\it support} $s(x)$  of  $x=x_1\circ \cdots \circ x_n\in X^*$ is
defined by \[s(x)=\{s(x_i): 1\leq i\leq n\}.\]
In particular, $s(\epsilon)=\emptyset$.
\end{defn}

Notice that when $s(x)$ is a singleton, we simply drop braces around it.
Below we use $[\,  ,\, ]$ for the equivalence class of a word under two different relations, so the reader should bear in mind the context in each case.

\begin{Prop}\label{lem:thedeletiontrick} Let $V'\subseteq V$ and let $\Gamma'=\Gamma(V',E')$ be the resulting full subgraph of $\Gamma$. Let $\mathscr{GP}'$ be the corresponding graph product of the monoids
$\mathcal{M}'=\{ M_\alpha:\alpha\in V'\}$.  Then $\mathscr{GP}'$ is a retract of $\mathscr{GP}$.
\end{Prop}
\begin{proof}  Let
$\eta:=\eta_{V,V'}:X^*\rightarrow \mathscr{GP}'$ be the morphism extending the map defined on $X$
 by
\[x\eta=\left\{ \begin{array} {ll} [x]& s(x)\in V'\\
{}[\epsilon]&\mbox{else.}\end{array}\right.\]
 We show that $R^{\sharp}\subseteq \ker \eta$.

First, for any $\alpha\in V$, whether or not $\alpha\in V'$, we have
$1_\alpha\eta=[\epsilon]=\epsilon \eta$ so that $R_{id}\subseteq \ker\eta$.

To see that $R_{v}\subseteq \ker\eta$, let $\alpha\in V$ and let $u, v\in M_\alpha$.   If $\alpha\not \in V'$, then
 \[(u\circ v)\eta=(u\eta)(v\eta)=[\epsilon][\epsilon]=[\epsilon]=(uv)\eta.\]
  If $\alpha\in V'$, then
  \[(u\circ v)\eta=(u\eta)(v\eta)=[u][v]=[u\circ v]=[uv]=(uv)\eta.\]

 Now consider $u\in M_\alpha, v\in M_\beta$ with $(\alpha, \beta)\in E$. If neither $\alpha$ nor $\beta$ is in $V'$, then
  $$(u\circ v)\eta=(u\eta)(v\eta)=[\epsilon][\epsilon]=(v\eta)(u\eta)=(v\circ u)\eta.$$
  If $\alpha, \beta\in V'$ with $(\alpha, \beta)\in E$, then, as $\Gamma'$ is a  full subgraph of $\Gamma$, we have $(\alpha, \beta)\in E'$, so that
  $$(u\circ v)\eta=(u\eta)(v\eta)=[u][v]=[u\circ v]=[v\circ u]=[v][u]=(v\eta)(u\eta)=(v\circ u)\eta.$$ If $\alpha\in V'$ but $\beta\not \in V'$ then
  $$(u\circ v)\eta=(u\eta)(v\eta)=[u][\epsilon]=[\epsilon][u]=(v\eta)(u\eta)=(v\circ u)\eta$$
 and dually if  $\alpha\not \in V'$ but $\beta\in V'$. Thus $R_e\subseteq \ker\eta$.

It follows that   $R^{\sharp}\subseteq \ker \eta$ and so
$\overline{\eta}:=\overline{\eta}_{V,V'}:\mathscr{GP}\rightarrow \mathscr{GP}'$ given by
$[w]\overline{\eta}=w\eta$ is a well defined morphism.

 It is easy to see  that $\iota:=\iota_{V',V}:  \mathscr{GP}'\rightarrow  \mathscr{GP}$
 such that $[w]\iota=[w]$ is well defined, and  by considering $\iota\eta$ it is clear that $\iota$ is an embedding.  It is then immediate that
 and $\eta\iota$ is a retraction of $\mathscr{GP}$ onto a submonoid $\mathscr{GP}'\iota$.
\end{proof}

We identify $\mathscr{GP}'$ with its image under $\iota$ and  regard $\mathscr{GP}'$ as a submonoid of $\mathscr{GP}$.

 \begin{Rem}\label{rem:retract} Let $\alpha\in V$. By taking $V'=\{\alpha\}$  in Proposition \ref{lem:thedeletiontrick}, we  immediately see that $M_\alpha$ is naturally embedded in $\mathscr{GP}$ via
 $\iota_\alpha:M_\alpha\rightarrow \mathscr{GP}$, where for $x\in M_\alpha$ we have $x\iota_\alpha= [x]$.
\end{Rem}

\begin{Prop} A graph product $\mathscr{GP}=\mathscr{GP}(\Gamma,\mathcal{M})$ is a direct limit of the graph products corresponding to the finite full subgraphs of $\Gamma$.
\end{Prop}
\begin{proof} The finite full subgraphs of $\Gamma$ are partially ordered by inclusion, and form a directed set under union. It is routine to see that the direct limit of the graph products $\mathscr{GP}'$, corresponding to finite full subgraphs with vertex set $V'\subseteq V$ and embeddings $\iota_{V',V''}$ where $V'\subseteq V''$, is isomorphic to
$\mathscr{GP}$.
\end{proof}

We end this subsection by remarking that  there are universal approaches to describe graph products of monoids as indicated in \cite[Proposition 1.6]{fountain:2009}, in the same way as there are for direct and free products.

\subsection{Regular, abundant and Fountain monoids}

We will denote the set of idempotents of a monoid $M$ by
$E(M)$.
We recall that Green's relations $\ar$ is defined on $M$ by the rule $a\, \ar\, b$ if and only if $aM=bM$. Equivalently, $a=bt$ and $b=as$ for some $s,t\in M$, thus, $\ar$ is a relation of mutual divisibility.  The relation $\el$ is defined dually.
It is easy to see that $M$ is regular if and only if every $a\in M$ is $\ar$-related to an idempotent, and so from considerations of duality, if and only if every $a\in M$ is $\el$-related to an idempotent. Graph products do not behave well with regard to regularity. Let $M$ and $N$ be  regular monoids containing elements $m,n$ respectively which do not have one-sided inverses.  Then  $[m\circ n]$ is not regular in the graph product $\mathscr{GP}(\Gamma,\mathcal{M})$ where
$\Gamma=(\{ 1,2\},\emptyset)$ and $\mathcal{M}=\{ M_1,M_2\}$ (that is, in the free product). See \cite{dacosta:2001} for a discussion of regularity in graph products. We therefore consider  relations larger than $\ar$ and  $\el$
and ask whether they contain idempotents.

The relation $\ars$ on a monoid $M$ was first defined in \cite{mcalister:1976,pastijn:1975}. For elements $a,b\in M$ we have
$a\,\ars\, b$ if and only if $a\,\ar\, b$ in some over-monoid $N$ of $M$.
Equivalently, for any $x,y\in M$ we have
\[xa=ya\mbox{ if and only if }xb=yb.\]
Thus, $\ars$ is a relation of mutual cancellativity. A third equivalent condition is that the principal left ideals $Ma$ and $Mb$ are isomorphic under a left ideal isomorphism where  $a\mapsto b$ \cite{fountain:1982}.
It is easy to see that $\ar\subseteq \ars$ with equality if $M$ is regular.  The relation  $\els$ is the left-right dual of  $\ars$.

\begin{defn}\label{defn:abundant} A monoid $M$ is {\em left abundant} if every element in $M$ is $\ars$-related to an idempotent. The notion of {\em right abundant} is defined dually, and $M$ is {\em abundant} if it is both left and right abundant.
\end{defn}

Examples of (left) abundant monoids abound; regular monoids are, of course, abundant; for a favourite non-regular example  take the  monoid $M_n(\mathbb{Z})$ of $n\times n$ integer matrices under matrix multiplication \cite{fountain:square}.

\begin{Rem}\label{rem:a} It is easy to see that for $a\in M$ and $e\in E(M)$ we have that
$a\,\ars\, e$ if and only if $ea=a$ and for any $x,y\in M$
\[xa=ya\Rightarrow xe=ye.\]\end{Rem}

A  monoid $M$ is {\em right cancellative} if for all $a,b,c\in M$, from $ac=bc$ we deduce that $a=b$; {\em left cancellative} is dual and $M$ is {\em cancellative} if it is right and left cancellative. It is easy to see that $M$ is right cancellative if and only if it is a single $\ars$-class. Thus, a right cancellative monoid is left abundant. A right cancellative monoid has no non-identity idempotents, and need not be left cancellative.  It follows that left abundancy does not imply right abundancy, which  contrasts with the case for regularity.

The relation $\art$ arose from many sources, as indicated in the Introduction. It extends the relation $\ars$  and coincides with it  in the case where the monoid is left  abundant. For elements $a,b$ of a monoid $M$ we have that
\[a\,\art\, b\mbox{ if and only if }ea=a\Leftrightarrow eb=b\mbox{ for all }e\in E(M).\]
The relation $\elt$ is defined dually.

\begin{defn}\label{defn:wabundant} A monoid $M$ is  {\em left Fountain} if every element in $M$ is $\art$-related to an idempotent. The notion of
{\em right Fountain} is defined dually, and $M$ is {\em Fountain} if it is both left and right Fountain.
\end{defn}

\begin{Rem}\label{rem:f} Similarly to Remark~\ref{rem:a}, it is easy to see that for $a\in M$ and $e\in E(M)$ we have that
$a\,\art\, e$ if and only if $ea=a$ and for any $f\in E(M)$
\[fa=a\Rightarrow fe=e.\]\end{Rem}

Formerly, left Fountain was referred to as {\em weakly left abundant}, but in view of the perceived significance  the notion was renamed by Margolis and Steinberg in \cite{margolis:2017}. It is easy to see that $M$ is left Fountain if and only if for any $a\in M$ the intersection of the principal, idempotent generated,  right ideals containing $a$ is principal and idempotent generated. As for abundancy, there are many natural examples of (non-abundant) (left) Fountain semigroups.
These include  finite monoids such that every principal (left) ideal has at most
one idempotent generator, for instance, any finite monoid with commuting idempotents \cite{margolis:2017}.  For some recent examples of Fountain monoids, consisting of semigroups of tropical matrices, see \cite{gould:2019}.

\begin{Rem} The relation   $\ar$ on a monoid $M$ is easily seen to be a  left  congruence, for any $a,b,c\in M$, if $a\,\ar\,  b$  then $ca\,\ar\, cb$. Similarly, $\ars$ is a left congruence. The same is not true, in general, for $\art$, even for some quite natural monoids (see, for example,  \cite[Proposition 6.10]{gould:2019}). Thus we do not assume that $\art$ is a left congruence in our calculations.
\end{Rem}

\section{(left) Foata normal forms}\label{sec:lfnf}

Throughout we let  $\mathscr{GP}=\mathscr{GP}(\Gamma,\mathcal{M})$ and follow the notation as established in  Section~\ref{sec:preliminaries}.
We show that elements in $\mathscr{GP}$ may be written in a normal form we refer to as
left Foata normal form. Such normal forms were previously known for elements of graph monoids,
 that is, where all the vertex monoids are free monogenic. The existing proofs rely on cancellativity, which is not available to us. Moreover, the presence of units in our vertex monoids provides an added complication.

\begin{defn}\label{defn:reduction}
Let $x_1\circ \cdots \circ x_n\in X^*$. A {\it reduction}  step is one of:

(id) $x_1\circ \cdots \circ x_n\rightarrow x_1\circ \cdots x_{i-1}
\circ x_{i+1}\circ\cdots \circ x_n$ where  $x_i\in I$;

(v) $x_1\circ \cdots \circ x_n\rightarrow x_1\circ \cdots x_{i-1}
\circ x_ix_{i+1}\circ x_{i+2}\circ\cdots \circ x_n$ where  $x_i, x_{i+1}\in M_\alpha$ for some $\alpha\in V$.

A {\it shuffle}  is a step:

(e) $x_1\circ \cdots \circ x_n\rightarrow x_1\circ \cdots \circ x_{i-1}
\circ x_{i+1}\circ x_i\circ x_{i+2}\circ\cdots \circ x_n$ where  $(s(x_i),s(x_{i+1}))\in E$.
\end{defn}

\begin{defn}\label{defn:shuffle} Two words in $X^*$ are {\it shuffle equivalent} if one can be obtained from the other by applying relations in $R_e$, or, equivalently, by shuffle steps.\end{defn}

\begin{defn}\label{def of reduced}
A word $x=x_1\circ \cdots \circ x_n\in X^*$ is {\em pre-reduced} if it is not possible to apply a reduction step to $x$.

A word $x=x_1\circ \cdots \circ x_n\in X^*$ is   {\it reduced} if for all $1\leq i\leq n$, $x_i\not \in I$, and for all $1\leq i< j\leq n $ with  $s(x_i)=s(x_j)$, there exists some $i<k<j$ with $(s(x_i), s(x_k))\not \in E$.

We denote by $K$ the set of reduced words in $X^*$.
\end{defn}

If $x=x_1\circ \cdots \circ x_n\in X^*$ is reduced, then
any factor $x_i\circ x_{i+1}\circ\cdots\circ x_j$ is reduced. A reduced word is pre-reduced, but the converse is not necessarily true. For example,
$x_1\circ x_2\circ x_3$ where $s(x_1)=s(x_3)=\alpha$, $s(x_2)=\beta$,  $(\alpha,\beta)\in E$ and no $x_i$ is an identity, is pre-reduced, but not  reduced. Notice that $\epsilon$ is always reduced.
The following remarks are clear from Definition~\ref{def of reduced}.

\begin{Rem}\label{rem:reduced} A word is  reduced  if and only if any word shuffle equivalent is pre-reduced. In particular, any word shuffle equivalent to a  reduced word is reduced.
\end{Rem}

\begin{Rem} \label{observations for Fountainicity} Let $x=x_1\circ \cdots \circ x_n$  and $y=y_1\circ \cdots \circ y_n\in X^*$ be such that $x_i, y_i\not \in I$ and $s(x_i)=s(y_i)$  for all $1\leq i\leq n$. If one of  $x,x^r,y, y^r$ is  reduced, then so are all four.
\end{Rem}

We will frequently concatenate reduced words in $X^*$, wanting to know if the product is reduced. The next remark is useful in this regard.

\begin{Rem}\label{cor:concatenate} Let $x=x_1\circ\cdots \circ x_m,\, y=y_1\circ\cdots \circ y_n\in X^*$ be reduced. Then $x\circ y$ is {\em not}
 reduced  exactly if there exists
$i,j$ with $1\leq i\leq m,1\leq j\leq n$ such that
$s(x_i)=s(y_j)$ and for all $h,k$ with $i<h\leq m,1\leq k<j$
we have $(s(x_i),s(z))\in E$ where $z=x_h$ or $z=y_k$.
\end{Rem}

The lemma below is standard  but it is  worth making explicit.

\begin{Lem}\label{lem:reduce} Let $w\in X^*$. Applying reduction steps and shuffles leads in a finite number of steps to a reduced word $\overline{w}$ with
$[w]=[\overline{w}]$.
\end{Lem}
\begin{proof}  Note that applying reduction steps to $w$ reduces its length.
There are finitely many words shuffle equivalent to $w$. Either these are all pre-reduced, and we
let $\overline{w}=w$, or we can apply a reduction step to some $w'$ shuffle equivalent to $w$.
Continue applying reduction steps to $w'$ until we arrive at a pre-reduced word $w_1$.  Notice that $|w_1|<|w|$. Repeat this process, obtaining a finite list of words $w=w_0, w_1,w_2,\hdots, w_m$
where all words shuffle equivalent to $w_m$ are pre-reduced. By Remark~\ref{rem:reduced},
$w_m$ is  reduced; let  $\overline{w}=w_m$.
\end{proof}

The next result is fundamental to our arguments.
As commented in \cite{fountain:2009}, it is the monoid version of Theorem 3.9 of Green \cite{Green:1990} (which can be applied directly to monoids). It can also  be deduced from \cite[Theorem 6.1]{dacosta:2001}; the reader should note that \cite{dacosta:2001}  uses different terminology to ours. However, we note that \cite{Green:1990} and
\cite{dacosta:2001} deal only with the case of a finite graph. Here we give the general result, calling upon Proposition~\ref{lem:thedeletiontrick}.

\begin{Prop}\label{shuffle}
Every element of $\mathscr{GP}$ is represented by a reduced word. Two reduced words represent the same element of $\mathscr{GP}$ if and only if they are shuffle equivalent.
An element $x\in [w]$ is of minimal length if and only if it is reduced.
\end{Prop}
\begin{proof} We have already shown the first part.

For the second, it is clear that if two reduced forms are shuffle equivalent then they represent the same element of $\mathscr{GP}$.
Conversely, suppose that $w,w'\in X^*$ are reduced forms and $[w]=[w']$ in
$\mathscr{GP}$. Let
$V'=s(w)\cup s(w')$ and let $\Gamma'=(V',E')$ be the corresponding full subgraph.
Let $X'=\bigcup_{\alpha\in V'}M_\alpha$ and let $\mathscr{GP}'$ be the corresponding graph product.
Clearly, $w,w'\in (X')^*$ are pre-reduced and from Proposition~\ref{lem:thedeletiontrick},
$[w]=[w']$ in
$\mathscr{GP}'$.  Theorem 1.1 of \cite{fountain:2009}, which may be deduced directly from original case for groups in \cite{Green:1990},  now tells us that that $w$ and $w'$ are shuffle equivalent in $\mathscr{GP}'$ and hence clearly shuffle equivalent in $\mathscr{GP}$.

For the final point, it is clear that a word $w\in X^*$ such that $|w|$ is minimal in $[w]$ is a reduced form. For the converse, suppose that $x\in X^*$ is a reduced form and  $[x]= [y]$. Choosing
$\overline{y}$ as in Lemma~\ref{lem:reduce} we have that $[x]=[y]=[\overline{y}]$ where
$\overline{y}$ is  reduced and $|y|\geq |\overline{y}|$. By the above $x,\overline{y}$ are shuffle equivalent and hence $|x|=|\overline{y}|\leq |y|$.
\end{proof}

\begin{defn}\label{defn:reducedform} If $x\in X^*$ and $[x]=[w]$ for a reduced word $w\in X^*$, then we say that $w$ is a {\em reduced form} of $x$.
\end{defn}

Notice that:

\begin{enumerate}\item  The equality $[x]=[y]$ where $x=x_1\circ\cdots \circ x_n$ and $y=y_1\circ \cdots \circ y_m$, does not, in general, imply that $s(x)=s(y)$. However, if both $x$ and $y$ are reduced, we must have $m=n$ and $s(x)=s(y)$, by Proposition~\ref{shuffle}.

\item  If $x_1\circ\cdots \circ x_n$ is  reduced  and $s(x_1\circ\cdots \circ x_n)$ is a complete subgraph,
then $s(x_i)\neq s(x_j)$ for all $1\leq i<j\leq n$, again by Proposition~\ref{shuffle}.
\end{enumerate}

We now show that, starting with a reduced word $x\in X^*$, and multiplying by a single letter $p$ from $X$, we have a narrow range of possibilities for
 any reduced form of  the product  $p\circ x$.

\begin{lemma}\label{product reduction}
Let $p\in X$, where $p\not \in I$,  and let $x=x_1\circ \cdots \circ x_n\in X^*$ be  reduced. Then one of the following occurs:

{\rm(i)}  $p\circ x_1\circ \cdots \circ x_n$ is  reduced;

{\rm(ii)} there exists $1\leq k\leq n$ such that $s(x_k)=s(p)$ and $(s(p), s(x_l))\in E$ for all $1\leq l\leq k-1$, and  $p\circ x_1\circ \cdots \circ x_n$ reduces to
\begin{equation}\label{eqn1} px_{k}\circ x_1\circ \cdots \circ x_{k-1}\circ x_{k+1}\circ \cdots \circ x_n\end{equation} and also to
\begin{equation}\label{eqn2}x_1\circ \cdots \circ x_{k-1}\circ px_k \circ x_{k+1}\circ \cdots \circ x_n.\end{equation}

Further, in Case {\rm(ii)}

{\rm (a)} if  $px_k$ is  not an identity  then (\ref{eqn1}) and (\ref{eqn2}) are reduced;

{\rm (b)} if  $px_k$ is  an identity then $p\circ x_1\circ \cdots \circ x_n$ reduces to the reduced word
\begin{equation}\label{eqn3} x_1\circ \cdots \circ x_{k-1} \circ x_{k+1}\circ \cdots \circ x_n.\end{equation}

Consequently, if $\alpha\in s(x)$ and $q\in X$ with $s(q)\neq \alpha$, then $\alpha$ must be in the support of any reduced form of $q\circ x$.
\end{lemma}

\begin{proof}
Suppose that $p\circ x$ is not  reduced. Then, by the definition of reduced, $k$ as defined in the statement must exist. Clearly, for $p\circ x$, we may shuffle $x_k$ and glue it to $p$ to obtain $$px_{k}\circ x_1\circ \cdots \circ x_{k-1}\circ x_{k+1}\circ \cdots \circ x_n$$ which is shuffle equivalent to \[x_1\circ \cdots \circ x_{k-1}\circ px_k \circ x_{k+1}\circ \cdots \circ x_n.\]
If $px_k$ is not an identity, then these words are reduced, by Remark~\ref{observations for Fountainicity}.

If $px_k$ is an identity then  $p\circ x$ reduces to \[x_1\circ \cdots x_{k-1}\circ x_{k+1}\circ \cdots \circ x_n,\]
which is reduced, since it is a right factor of the word
$x_k\circ x_1\circ \cdots x_{k-1}\circ x_{k+1}\circ \cdots \circ x_n$, which is shuffle equivalent to the reduced word $x$.

 The final statement is clear if $q\in I$; if $q\notin I$ it  follows by examining the cases above.
\end{proof}

\begin{Cor}\label{product of reduced}
Let $x,y \in X^*$ where $y$ is reduced. If $\alpha\in s(y)$ but $\alpha\not \in s(x)$, then $\alpha$ must be in the support of any reduced form of $x\circ y$.
\end{Cor}

 \begin{proof} Let $x=x_1\circ\cdots \circ x_m$ and  proceed by induction on $m$. If $m=1$ then the result is true by Lemma \ref{product reduction}. Suppose therefore that $m\geq 2$ and the result is true for $m-1$.  Let $z_1\circ\cdots\circ z_k$ be a reduced form of $x_2\circ \cdots \circ x_m\circ y$. Then $\alpha$ is in the support of  $z_1\circ\cdots\circ z_k$ by assumption, and so $\alpha$ is in the support of the reduced form of $x_1\circ z_1\circ \cdots \circ z_k$ and hence $x\circ y$, again  by Lemma \ref{product reduction}.
\end{proof}

We will make extensive use of Corollary~\ref{product of reduced} to find reduced forms of products of  reduced words.
The expression of elements in $\mathscr{GP}$ using reduced forms has a very useful cancellation-type property, as we now explain. First, another
technical result using a strategy that will be key in this paper. Recall from Definition~\ref{def of reduced} that
$K=\{w\in X^*: w\mbox{ is reduced}\}$.

\begin{lemma}\label{smallest or largest} Let  $\alpha\in V$ and define maps
$$\theta_\alpha: K\longrightarrow \mathscr{GP}\mbox{ and }\eta_\alpha:  K\longrightarrow \mathscr{GP}$$
where for each $x=x_1\circ \cdots \circ x_n\in K$,
\[x\theta_{\alpha}=\left\{ \begin{array} {ll} [x^{i(\alpha)}]& \alpha\in s(x)\\

[ \epsilon]&\mbox{else}\end{array}\right. \mbox{ and }  x\eta_{\alpha}=\left\{ \begin{array} {ll} [x_{i(\alpha)}]& \alpha\in s(x)\\

[\epsilon]&\mbox{else.}\end{array}\right.\]
Here $i(\alpha)$ is  the smallest $i$ such that $s(x_i)=\alpha$ and  $x^{i(\alpha)}$ is obtained by deleting $x_{i(\alpha)}$ from $x$. Then
$\theta_\alpha$ and $\eta_\alpha$ are constant on $R^{\sharp}$-classes, that is, they extend to  maps
$$\overline{\theta}_{\alpha}: \mathscr{GP}\longrightarrow \mathscr{GP}\mbox{ and }\overline{\eta}_{\alpha}: \mathscr{GP}\longrightarrow \mathscr{GP}$$
given by
\[ [w]\overline{\theta}_{\alpha}= w'\theta_\alpha\mbox{ and }  [w]\overline{\eta}_{\alpha}= w'\eta_\alpha\] where $w'$ is {\em any} reduced form of $w$.
\end{lemma}

\begin{proof}
Let $[p]=[q]$ where both $p,q\in K$ are reduced. We need show $p\theta_\alpha=q\theta_\alpha$ and $p\eta_\alpha=q\eta_\alpha$. By Proposition~\ref{shuffle}, $p$ and $q$ are shuffle equivalent; by finite induction  we can  assume that $q$ is obtained from $p$ by exactly one shuffle.

 Let
$$p=x_1\circ \cdots \circ x_{j-1}\circ x_{j}\circ x_{j+1}\circ x_{j+2}\circ \cdots \circ x_n$$ and $$q=x_1\circ \cdots \circ x_{j-1}\circ x_{j+1}\circ x_{j}\circ x_{j+2}\circ \cdots \circ x_n.$$

If $\alpha\not \in s(p)$ (and so $\alpha\not \in s(q)$), then
$$p\theta_\alpha=[p]=[q]=q\theta_\alpha\mbox{ and } p\eta_\alpha=[\epsilon]=q\eta_\alpha.$$

Suppose now that $\alpha\in s(p)$. Considering $p$, pick the smallest $k$ such that $s(x_k)=\alpha$. If $1\leq k\leq j-1$ or $j+2\leq k\leq n$, then, clearly, $p\theta_\alpha=q\theta_\alpha$ and $p\eta_\alpha=q\eta_\alpha$.  If $k=j$, then since $(s(x_j), s(x_{j+1}))\in E$ we have
$s(x_j)\neq s(x_{j+1})$; it follows that $p\eta_{\alpha}=q\eta_{\alpha}=[x_j]$ and  $p\theta_{\alpha}=q\theta_{\alpha}=
[x_1\circ \cdots \circ x_{j-1}\circ x_{j+1}\circ x_{j+2}\circ \cdots \circ x_n]$. Similarly  if $k=j+1$.

\end{proof}

It is useful to state the dual of Lemma \ref{smallest or largest}.

\begin{lemma}\label{smallest or largest-2} Let  $\alpha\in V$ and define maps
$$\delta_\alpha: K\longrightarrow \mathscr{GP}\mbox{ and }\tau_\alpha:  K\longrightarrow \mathscr{GP}$$
where for each $x=x_1\circ \cdots \circ x_n\in K$,
\[x\delta_{\alpha}=\left\{ \begin{array} {ll} [x^{j(\alpha)}]& \alpha\in s(x)\\

[x]&\mbox{else}\end{array}\right. \mbox{ and }  x\tau_{\alpha}=\left\{ \begin{array} {ll} [x_{j(\alpha)}]& \alpha\in s(x)\\

[\epsilon]&\mbox{else.}\end{array}\right.\]
Here $j(\alpha)$ is the largest $j$ such that $s(x_j)=\alpha$ and  $x^{j(\alpha)}$ is obtained by deleting $x_{j(\alpha)}$ from $x$. Then
$\delta_\alpha$ and $\tau_\alpha$ are constant on $R^{\sharp}$-classes, that is, they extend to  maps
$$\overline{\delta}_{\alpha}: \mathscr{GP}\longrightarrow \mathscr{GP}\mbox{ and }\overline{\tau}_{\alpha}: \mathscr{GP}\longrightarrow \mathscr{GP}$$
given by
\[ [w]\overline{\delta}_{\alpha}= w'\delta_\alpha\mbox{ and }  [w]\overline{\tau}_{\alpha}= w'\tau_\alpha\] where $w'$ is {\em any} reduced form of $w$.
\end{lemma}

We use the maps defined in Lemmas \ref{smallest or largest} and \ref {smallest or largest-2} to prove our first cancellation-type result.

\begin{lemma}\label{block}
Let $[x]=[y]$ where  $x=x_1\circ \cdots \circ x_n$ and $y=y_1 \circ \cdots \circ y_n$ are reduced and let $1\leq m\leq n$. Then $[x_1\circ \cdots \circ x_m]=[y_1\circ \cdots \circ y_m]$ if and only if $[x_{m+1}\circ \cdots \circ x_n]=[y_{m+1}\circ \cdots \circ y_n]$.
\end{lemma}

\begin{proof}
Suppose that  $[x_1\circ \cdots \circ x_m]=[y_1\circ \cdots \circ y_m]$.  Since  $[x_1\circ \cdots \circ x_n]=[y_1 \circ \cdots \circ y_n]$, we have
$$[x_1\circ \cdots x_{m}\circ x_{m+1}\circ \cdots \circ x_n]=[x_1\circ \cdots x_{m}\circ y_{m+1}\circ \cdots \circ y_n].$$
As $x_1\circ \cdots x_{m}\circ x_{m+1}\circ \cdots \circ x_n$ is  reduced  and $x_1\circ \cdots x_{m}\circ y_{m+1}\circ \cdots \circ y_n$ has the same length, we deduce that $x_1\circ \cdots x_{m}\circ y_{m+1}\circ \cdots \circ y_n$ is  reduced, by Proposition~\ref{shuffle}. Let $s(x_r)=\alpha_r$ for all $1\leq r\leq m$. Then, observing that any right factor of a reduced word is reduced,
 $$[x_1\circ \cdots \circ x_{m}\circ x_{m+1}\circ \cdots \circ x_n]\overline{\theta}_{\alpha_{1}}\cdots \overline{\theta}_{\alpha_{m}}=[x_1\circ \cdots \circ x_{m}\circ y_{m+1}\circ \cdots \circ y_n]\overline{\theta}_{\alpha_{1}}\cdots \overline{\theta}_{\alpha_{m}}$$ by  Lemma \ref{smallest or largest}, which gives  $[x_{m+1}\circ \cdots \circ x_n]=[y_{m+1}\circ \cdots \circ y_n]$ as desired.

The remainder of the lemma follows dually from Lemma \ref{smallest or largest-2}, by applying the maps $\overline{\delta_\alpha}$.
\end{proof}

\begin{defn} \label{defn:complete} A word $w\in X^*$ is a  {\it complete block} if it is reduced, and $s(w)$ forms a complete subgraph of $\Gamma=\Gamma (V, E)$.
\end{defn}

We now show that any reduced word in $X^*$  may be shuffled  into a word that is a  product of complete blocks.

\begin{defn}\label{defn:lcrf} Let $w\in X^*$. Then $w$ is a {\it left Foata normal form}
with {\em block length $k$} and {\em blocks} $w_i\in X^*$, $1\leq i\leq k$, if:

(i) $w=w_1\circ \cdots \circ w_k\in X^*$ is a reduced word;

(ii) $s(w_i)$ is a complete subgraph for all $1\leq i\leq k$;

(iii) for any $1\leq i<k$ and  $\alpha\in s(w_{i+1})$, there is some $\beta\in s(w_i)$ such that $(\alpha, \beta)\not\in E$.
\end{defn}

If $[x]=[w]$ where $w$ is a left  Foata normal form, then we may say
$w$ is a  {\em  left  Foata normal form of $x$}.

 \begin{Rem}\label{rem:lfnf} (i)
The empty word  $\epsilon$ is a left Foata normal form with block length 0.  (ii) A complete block is precisely a word in left  Foata normal form with block length $1$. (iii) If $w=w_1\circ \cdots \circ w_k\in X^*$ is in left Foata normal form with blocks $w_i$, $1\leq i\leq k$, then
 for any $1\leq j\leq j'\leq k$ we have
 $w_j\circ w_{j+1}\circ\cdots\circ w_{j'}$ is also in left Foata normal form, with blocks $w_h$, $j\leq h\leq j'$.
\end{Rem}

\begin{Prop}\label{lem:leftreducedform}
Every element in $\mathscr{GP}$ may be represented by a left Foata normal form.
\end{Prop}

\begin{proof}  We know that any element of $\mathscr{GP}$ may be represented by a reduced word. Take a reduced word $w=y_0$ and let
 $w_1$ be chosen  such that
$w_1\circ y_1$ is shuffle equivalent to $w$ for some $y_1$, $s(w_1)$ is complete, and
$|w_1|$ is maximum with respect to these constraints. Assume that
$w_1,y_1,w_2,y_2,\hdots, w_k,y_{k}$ have been chosen such that
for each $1\leq j\leq k$ we have that $y_{j-1}$ is shuffle equivalent to $w_j\circ y_{j}$, $s(w_{j})$ is complete, and $|w_{j}|$ is maximum with respect to these constraints. Clearly this process must end after a finite number of steps with $y_k=\epsilon$.

 For any $1\leq j\leq k$ we have by finite induction
that $y_{j-1}$ is shuffle equivalent to $w_{j}\circ w_{j+1}\circ\cdots\circ w_k$ and, in particular,  $w$ is shuffle equivalent to $w_1\circ \cdots \circ w_k$. We now claim that $w_1\circ \cdots \circ w_k$ is a left Foata normal form with blocks $w_i$ for $1\leq i\leq k$.
Certainly (i) and (ii) of Definition~\ref{defn:lcrf} hold. To see that (iii) holds, suppose that
$1\leq i<k$ and let $\alpha\in s(w_{i+1})$; say $w_{i+1}=p\circ a\circ q$ where $a\in X$ and
$s(a)=\alpha$. Suppose for contradiction that  for all $\beta\in s(w_i)$ we  have
$(\alpha,\beta)\in E$.  Since $y_{i-1}$ is shuffle equivalent to
$w_{i}\circ w_{i+1}\circ\cdots\circ w_k$  we would have
$y_{i-1}$ being shuffle equivalent to
$w_{i}\circ a\circ y_{i+1}'$ for some $y_{i+1}'$, where
$s(w_{i}\circ a)$ is complete and $|w_{i}\circ a|>|w_{i}|$, a contradiction.
\end{proof}

\begin{Rem}\label{ttt}
Let $x=x_1\circ \cdots \circ x_n$ and $z=z_1\circ \cdots \circ z_n$ be reduced forms of $w$. Pick $\alpha\in s(x)(=s(z))$. Let $i$ be least such that $s(x_i)=\alpha$ and $j$ be least such that $s(z_j)=\alpha$.  Since $x$ and $z$ are shuffle equivalent,  $x_i=z_j$. Suppose that there exists some $1\leq i'< i$ such that $s(x_{i'})=\beta$ with $(\beta, \alpha)\not \in E$; note that by minimality of $i$ we have $\beta\neq \alpha$. Then, again as $x$ and $z$ are shuffle equivalent, there exists some $1\leq j'<j$ such that  $s(z_{j'})=\beta$ and $z_{j'}=x_{i'}$.
\end{Rem}

We are now in a position to prove the main result of this section, which tells us that the left  Foata normal form of an element of any $\mathscr{GP}$ is essentially unique.

\begin{Thm}\label{thm:uniqueness} Let $w\in X^*$ and
let $w_1\circ w_2\circ\cdots\circ w_k$ and $w_1'\circ w_2'\circ\cdots \circ w_h'$ be left Foata normal forms of $w$ with blocks $w_i,w_j'$ for $1\leq i\leq k, 1\leq j\leq h$.  Then $k=h$ and $[w_i]=[w_i']$
for $1\leq i\leq k$.
\end{Thm}
\begin{proof}
Let $p_1=w_2\circ\cdots\circ w_k$ and
$p_1'=w_2'\circ\cdots \circ w_h'$; by Remark~\ref{rem:lfnf}
$p_1$ and $p_1'$ are also in left Foata normal form.  We claim that $s(w_1)=s(w_1')$. Expressing as products of letters, let
$$w_1=a_1\circ \cdots a_r, \, p_1=b_1\circ \cdots \circ b_m, \, w_1'=u_1\circ\cdots \circ u_t \mbox{~and~}p_1'=v_1\circ \cdots \circ v_n.$$

Suppose that there exists some $\delta\in s(w_1)$ but not in $s(w_1')$, so that  $\delta\in s(p_1')$. Let $i$ be least such that $s(a_i)=\delta$ and let  $j$ be least such that $s(v_j)=\delta$. By definition of left Foata normal form, either (i) $v_j$ is in the first block $w_2'$ of $p_1'$, in which case there exists some $1\leq t'\leq t$ with $(s(u_{t'}), \delta)\not \in E$,
or (ii) $v_j$ is in a subsequent  block of $p_1'$ in which case certainly  there exists $1\leq j'<j$ with
 $(s(v_{j'}), \delta)\not \in E$.  Let $\gamma=s(u_{t'})$ (in Case (i)) and
 $\gamma=s(v_{j'})$ (in Case (ii)). In either case we have  $\gamma\neq \delta$ and $(\delta,\gamma)\notin E$. By Remark \ref{ttt} there must be
 some $i'$ with $1\leq i'<i$ such that  $s(a_{i'})=\gamma$. This is impossible since
 $s(w_1)$ is a complete subgraph. Together with the converse argument we  deduce that  $s(w_1)=s(w_1')$.

We now show that $[w_1]=[w_1']$ and $[p_1]=[p_1']$. Let $s(w_1)=\{\alpha_1, \cdots, \alpha_r\}$. It then follows from Lemma~\ref{smallest or largest} that
$$[p_1]=[w_1\circ p_1]\overline{\theta}_{\alpha_1}\cdots \overline{\theta}_{\alpha_r}=[w_1'\circ p_1']\overline{\theta}_{\alpha_1}\cdots \overline{\theta}_{\alpha_r}=[p_1']$$ and
$$[w_1]=[w_1\circ p_1]\overline{\eta}_{\alpha_1}\cdots [w_1\circ p_1]\overline{\eta}_{\alpha_r}=[w_1'\circ p_1']\overline{\eta}_{\alpha_1}\cdots [w_1'\circ p_1']\overline{\eta}_{\alpha_r}=[w_1']$$ as required.

Noticing that $|p_1|<|w_1\circ p_1|$, the result now follows by induction.
\end{proof}

Clearly,  we may define the notion of a {\it right Foata normal form} of an element in $X^*$, and the dual arguments to those for left Foata normal form hold.

\section{ Towards a characterization of $\mathcal{R}^*$}\label{sec:gp}

We continue to consider a fixed, but arbitrary, graph product of monoids $\mathscr{GP}$. We now show how we can use the left Foata normal forms developed in Section~\ref{sec:lfnf}
to describe  the relation $\mathcal{R}^*$ in $\mathscr{GP}$.  We will build on this in Section~\ref{sec:abundancy}   to show that if each vertex monoid is abundant, then so is $\mathscr{GP}$.

 The next lemma can be deduced from \cite[Proposition 7.1]{dacosta:2001}, together with our Proposition~\ref{lem:thedeletiontrick} and Remark~\ref{rem:retract}. Note that if $x=x_1\circ\cdots \circ x_n$
  is  reduced, then in Costa's terminology, the $x_i$ are {\em components}.

\begin{lemma}\label{invertible}
Let $x=x_1\circ\cdots \circ x_n\in X^*$ be  reduced. Then
the following are equivalent:

\begin{enumerate}\item  $[x]$ is left invertible in $\mathscr{GP}$;
\item  $[x_i]$ is left invertible in $\mathscr{GP}$ for  $1\leq i\leq n$;
\item $x_i\in M_{s(x_i)}$ is left invertible in $M_{s(x_i)}$ for  $1\leq i\leq n$.
\end{enumerate}

Moreover, if any of the above conditions hold, then
any left inverse of $[x]$  has the form $[y]$ where
$y= y_n\circ\cdots \circ y_1$ and $y_i$ is a left inverse of $x_i$ for $1\leq i\leq n$.
\end{lemma}

The arguments in the next lemma essentially rely on the following simple observations. If
$x=x_1\circ\cdots\circ x_n\in X^*$ is shuffle equivalent to
$y=x_{j_1}\circ\cdots\circ x_{j_n}$, then for any
$1\leq i<k\leq n$ with $j_i>j_k$ we have $(s(x_{j_i}),s(x_{j_k}))\in E$.  Suppose we can shuffle $x$ to a
word $x'\circ x''$,  where $x'$ has length $m$. Consequent to the previous remark,  we can then
shuffle $x'$ to a word $x_{i_1}\circ x_{i_2}\circ\cdots\circ x_{i_m}$ where
$i_1<i_2<\cdots<i_m$ and $ x''$ to the word obtained from $x$ by deleting the letters
$x_{i_1},\cdots ,x_{i_m}$. Moreover, for any $1\leq \ell\leq m$ we can shuffle the letters $x_1,\cdots,  x_{i_{\ell}-1},x_{i_{\ell}}$ in $x'\circ x''$ back to the first $i_{\ell}$ positions, resulting in having shuffled
$x$ to $x_1\circ x_2\cdots\circ  x_{i_{\ell}-1}\circ x_{i_{\ell}}\circ x_{i_{\ell+1}}\circ\cdots \circ x_{i_m}\circ
z$ where $z$ is $x_{i_{\ell}+1}\circ x_{i_{\ell}+2}\circ\cdots \circ x_n$ with
$x_{i_{\ell+1}},\cdots, x_{i_m}$ deleted.

\begin{lemma}\label{decomposition}
Let $u\in X^*$.   Then:
\begin{enumerate}

\item $[u]=[a][x]$ where  $a\circ x$ is  reduced, $[a]$ is left invertible,  and $|a|$ is maximum with respect to these constraints;
\item with $[u]=[a][x] $  as in (1), if in addition $[a][x]=[b][y]$ where (in addition) $b\circ y$ is reduced, $[b]$ is left invertible,  and $|b|=|a|$, then $[a]=[b]$ and $[x]=[y]$;
\item  with $[u]=[a][x] $  as in (1), $x$ has a left Foata normal form  $x_1\circ \cdots \circ x_m$ with blocks $x_i$, $1\leq i\leq m$, such that
 $x_1$ contains no left invertible letters.
 \end{enumerate}
\end{lemma}

\begin{proof}  We begin by finding  a reduced form $p=p_1\circ\cdots\circ p_n$ for $u$. By shuffling $p$ we
 may find $a$ and $x$ as in (1). By Lemma~\ref{invertible} and the above remark we may assume that
$a=p_{i_1}\circ\cdots \circ p_{i_k}$ where
$i_1<i_2<\cdots< i_k$, with $p_{i_h}$ is left invertible for all $1\leq h\leq k$.

Suppose now that $b,y$ are as given; again we may assume that
$b=p_{j_1}\circ\cdots \circ p_{j_k}$ where
$j_1<j_2<\cdots< j_k$.
If $i_1<j_1$ then we notice that
we can shuffle $p$ to $p_{i_1}\circ p_1\circ\cdots \circ p_{j_1-1}\circ p_{j_1}\circ p_{j_1+1}\circ\cdots \circ p_n$
and then to $p_{i_1}\circ p_{j_1}\circ p_{j_2}\circ\cdots \circ p_{j_k}\circ y'$ where $y'$ is $y$ with
$p_{i_1}$ deleted. But, this contradicts the maximality of $|a|$. With the dual argument we obtain that
$i_1=j_1$.

Suppose for finite induction that $i_\ell=j_\ell$ for $1\leq \ell\leq s<k$  and that
$i_{s+1}<j_{s+1}$. Then similarly to the preceding argument we have that
$p$ shuffles to
$p_{i_1}\circ p_{i_2}\circ\cdots \circ p_{i_s}\circ p_{i_{s+1}}\circ z$ where  $z$ is
$p$ with $p_{i_1},p_{i_2},\cdots, p_{i_{s+1}}$ deleted. But then we can shuffle  $z$
to obtain a word $ p_{j_{s+1}}\circ p_{j_{s+2}}\circ\cdots\circ p_{j_k}\circ w$ where $w$ is $p$ with
$p_{i_1},p_{i_2},\cdots, p_{i_{s+1}}, p_{j_{s+1}},  p_{j_{s+2}}, \cdots, p_{j_k}$
deleted.  Again, this contradicts the maximality of
$|a|$. We deduce that $i_s=j_s$ for $1\leq s\leq k$ and hence $[a]=[b]$. Clearly  then  $[x]=[y]$ follows.

Suppose now that $[u]=[a][x]$ as in (1), and shuffle  $x$ to  left Foata normal form
 $x_1\circ \cdots \circ x_m$, where the $x_i$ are the blocks for $1\leq i\leq m$. Clearly, since
 $s(x_1)$ is complete, $x_1$ cannot contain any left invertible letters, else this would contradict the maximality of $|a|$.
\end{proof}

To simplify the description of $\mathcal{R}^*$ on $\mathscr{GP}$ we now present two technical lemmas.

\begin{lemma}\label{key-1}
Let $x=x_1\circ \cdots \circ x_n\in X^*$ and $(\alpha, \beta)\not\in E$. Suppose that    $x_l$ is non-left invertible with $s(x_l)=\beta$, for some $1\leq l\leq n$, and $s(x_k)$ is neither $\alpha$ nor $\beta$ for all $l<k\leq n$.  Let  $z=z_1 \circ \cdots \circ z_m\in X^*$ be any reduced form of $x_1\circ \cdots \circ x_n$. Then
 $\beta\in s(z)$ and if $j$ is greatest such that $1\leq j\leq m$ with $s(z_j)=\beta$, then  $z_j$ is non left invertible, and  $s(z_t)\neq \alpha$ for all $j<t\leq m$.
\end{lemma}
\begin{proof}  We begin by observing that if we can find one reduced form of $x$ with the required property, then all reduced forms will have the required property.

We proceed by induction on $n$. If $n=1=l$ the result is clear,
since $x=x_1$ is the only reduced form of $x$. Suppose now that $n>1$ and the result is true for all words of length strictly less than $n$.

Let $w_1= x_1\circ \cdots \circ x_{l-1}$, $w_2=x_{l+1}\circ \cdots \circ x_n$ and let
$w_1',w_2'\in X^*$ be reduced  such that $[w_1]=[w_1']$ and $[w_2]=[w_2']$.
Certainly  $\alpha,\beta\notin s(w_2')$.
Let $w_1'=u_1\circ \cdots\circ u_h$ and  $w_2'=v_1\circ \cdots\circ v_r$. If
$w_1'\circ x_l\circ w_2'$ is  a reduced form, then we are done.

Suppose therefore that $w_1'\circ x_l\circ w_2'$ is not a reduced form, and consider first
\[w_1'\circ x_l=u_1\circ \cdots\circ u_h\circ x_l.\]
If $w_1'\circ x_l$   is not a reduced form  then, from Remark~\ref{cor:concatenate}, there exists
some $t$ with $1\leq t\leq h$ with $s(u_t)=\beta$ and $(s(u_k),\beta)\in E$ for all $t<k\leq h$.
By shuffling $w_1'$, without loss of generality we can assume that $t=h$.
 Let $p=u_hx_l$ and notice  that as $x_l$ is not left invertible, then neither is $p$, and certainly $p\neq \epsilon$. Then
\[y=u_1\circ \cdots\circ u_{h-1}\circ p\circ v_1\circ \cdots\circ v_r\]
has length strictly less than $n$, $s(p)=\beta$, $p$ is not left invertible, and
$\alpha,\beta\notin s(v_1\circ\cdots\circ v_r)$.

On the other hand, if $w_1'\circ x_l$ is  a reduced form, then again by Remark~\ref{cor:concatenate}, and making use of the fact
$\beta\notin s(w_2')$, we may assume that $s(u_h)=s(v_1)$ and $(\beta,s(u_h))\in E$.
Then
\[y= u_1\circ \cdots\circ u_{h-1}\circ u_hv_1\circ x_l\circ v_2\circ\cdots\circ v_r\]
has length strictly less than $n$,  and
$\alpha,\beta\notin s(v_2\circ\cdots\circ v_r)$.

In each case we have  found a word $y$ with $[y]=[x]$ to which we can apply the induction hypothesis.  The result follows. \end{proof}

\begin{lemma}\label{key-2}
Let $x=x_1\circ \cdots \circ x_n$ be a left  Foata normal form
with blocks $x_i, 1\leq i\leq n$, such that $x_1$ contains no  left invertible letters. Let $u\in X^*$ and let $z$ be a reduced form of $u\circ x_1$. Then $z\circ x_2\circ \cdots \circ x_n$ is a reduced form of $u\circ x$.
\end{lemma}

\begin{proof}  Certainly $[u\circ x]=[z\circ x_2\circ \cdots \circ x_n]$.
Let $z=z_1\circ \cdots \circ z_m$.
As both $z$ and $x_2\circ \cdots \circ x_n$ are reduced, if  $z\circ x_2\circ \cdots \circ x_n$ is not  reduced, then by Remark~\ref{cor:concatenate}
we can shuffle a letter $z_k$ of $z$ to the end of $z$ and a letter $a$  of $x_2\circ \cdots \circ x_n$
to the start of $x_2\circ \cdots \circ x_n$ where $s(z_k)=s(a)=\alpha$ say. We may assume that $k=m$ and as
$x_2\circ \cdots \circ x_n$ is a left Foata normal form, that $a$ is a letter of $x_2$, and then that it is the first letter of $x_2$.
Since  $x$ is in left Foata normal form, it follows that  $\alpha\not \in s(x_1)$ and there exists
a (unique)  letter $b$ in $x_1$ such that  $(\alpha,s(b))\not\in E$. Let $s(b)=\beta$; recall that  $b$ is  non-left invertible. It then follows from Lemma \ref{key-1} that $\beta\in s(z)$ and  if
 $t$ is greatest such that $1\leq t\leq m$ with $s(z_t)=\beta$, then    $s(z_h)\neq \alpha$ for all $t<h\leq m$. This contradicts the fact $s(z_m)=\alpha$.

We deduce that  $z\circ x_2\circ \cdots \circ x_n$  is a reduced form of $u\circ x$, as required.
\end{proof}

 We can now get our first handle on the  consideration of the $\ars$-class of an element of $\mathscr{GP}$ in the general case. Subsequently, we will focus on the case where the vertex monoids are abundant.

\begin{Prop}\label{step-1}
\begin{enumerate} \item
Let $x=x_1\circ \cdots \circ x_n$ be a left Foata normal form with blocks $x_i$, $1\leq i\leq n$, such that $x_1$ contains no  left invertible letters. Then   $ [x]\, \mathcal{R}^*\, [x_1]$.
\item  Let $p\in X^*$. Then  $[p]=[a][x]$ where $a\circ x$ is reduced, the letters of $a$ are all left invertible,  $|a|$ is maximum with respect to these constraints and $x$  is a left Foata normal form $x$ as in (1). Further,  $[p]\,\ars\, [a][x_1]$.
\end{enumerate}
\end{Prop}

\begin{proof} (1) Let $[p],[q]\in \mathscr{GP}$. Clearly it suffices to show that if
$[p][x]=[q][x]$, then $[p][x_1]=[q][x_1]$.
Suppose therefore that $[p][x]=[q][x]$ and
let $(p\circ x_1)'$ and $(q\circ x_1)'$ be reduced forms of $p\circ x_1$ and $q\circ x_1$, respectively. By Lemma \ref{key-2}, $(p\circ x_1)'\circ x_2 \circ \cdots\circ  x_n$ and $(q\circ x_1)'\circ x_2\circ \cdots \circ x_n$ are reduced forms of $p\circ x_1\circ \cdots \circ x_n$ and $q\circ x_1\circ \cdots \circ x_n$, respectively. It then follows from Lemma \ref{block} that  $[(p\circ x_1)']=[(q\circ x_1)']$ and so $[p][x_1]=[q][x_1]$.

(2)  This existence of $a$ and $x$ is guaranteed by Lemma~\ref{decomposition}, and then the result follows from (1) and the fact that $\ars$ is a left congruence.
\end{proof}

\section{ Graph products of left abundant monoids are left abundant}\label{sec:abundancy}

 The aim of this section is to prove the claim of the heading; this will involve us in some combinatorial intricacies. It might be helpful to the reader if we outline our strategy here.
Proposition~\ref{step-1} is our first step in describing $\ars$ in $\mathscr{GP}$. In  Proposition~\ref{step-2} we show that if $z=z_1\circ \cdots \circ z_n\in X^*$ is a complete block, then $[z]~\mathcal{R}^*~[z']$ where $z'=z_1'\circ\cdots\circ z_n'$ is chosen such that  $z_i'\in M_{s(z_i)}$ and
$z_i\,\mathcal{R}^*\, z_i'$ in $M_{s(z_i)}$ for all $1\leq i\leq n$. In particular, if each $M_i$ is left abundant, then  for any idempotents $z_i^+$ with $z_i~\mathcal{R^*}~z_i^+$ in $M_{s(z_i)}$,
 we have that $[z]$ is $\mathcal{R}^*$-related to the  idempotent $[z^+]$ where $z^+=
 z_1^+\circ \cdots \circ z_n^+$. Proposition~\ref{step-1} tells us that for $p\in X^*$ we can write $[p]=[a][x]$ where $a\circ x$ is reduced, the letters of $a$ are all left invertible, and $x$  is a left Foata normal form, the first block of which contains no left invertible letters. Moreover, calling this first block $z$ we have that  $[p]\,\ars\, [a][z]$. As $\ars$ is a left congruence, $[p]\,\ars\, [a][z^+]$ and then
 as $[a]$  has a left inverse $[a']$ (so that $[a']\,\ar\, [\epsilon])$ we have $[p]\,\ars\, [a][z^+][a']$. The fact that $[a][z^+][a']$ is idempotent is easily seen.

To arrive at Proposition~\ref{step-2} we cannot escape a very careful analysis of  products $[x][z]$ in $\mathscr{GP}$ (remember, we are considering equations of the form $[x][z]=[y][z]$). To this end we find a new factorisation of elements in $\mathscr{GP}$ that allows us to cancel and replace a final term in equalities. This we achieve in Lemma~\ref{phi}.

To arrive at Lemma~\ref{phi} we now define the  notions of $\alpha$-absorbing, $\alpha$-good and subsequently a stronger version of being $\alpha$-good that we call $\alpha$-amenable, where $\alpha\in V$.
We show in Proposition~\ref{absorbed} that in an $\alpha$-amenable word, the inner factor reduces to a word which does not have $\alpha$ in its support. This enables us to pin down exactly which letters we can move
to the right of  a word (see   Definition ~\ref{def of N}) and hence we arrive at the factorisation of Lemma~\ref{phi}.

First,   we  need to recall the description of  idempotents in
$\mathscr{GP}$ from \cite{dacosta:2001}.

\begin{defn}\label{standard idempotent}
We say that an idempotent of $\mathscr{GP}$ is in {\em standard form} if  it is written as $[u]$ where $u=b\circ e\circ b'\in X^*$ is  reduced,
\[b=b_1\circ \cdots\circ b_n,\,\,
e= e_1\circ \cdots\circ e_m, b'= b_n'\circ\cdots \circ b_1'\]
where $b_i'b_i$ is an identity for $1\leq i\leq n$, $s(e)$ is complete and $e^2_i=e_i$ for $1\leq i\leq m$.
\end{defn}

Note that $[u]$ is idempotent for any word $u$ of the form in Definition
~\ref{standard idempotent}.

 \begin{lemma}\cite[Theorem 14.2]{dacosta:2001}\label{idempotent}
Any idempotent in $ \mathscr{GP}$ can be written in standard form.
\end{lemma}

\begin{defn}\label{defn of i-absorbed}
Let $\alpha\in V$. A word $x\in X^*$ is said to be $\alpha$-{\it absorbing} if $\alpha$ is not in the support of  any reduced form of $x$.
\end{defn}

\begin{defn}\label{defn of i-good}
Let $\alpha\in V$. A word $x\in X^*$ is said to be $\alpha$-{\it good} if for all $\beta$ in the support of
any reduced form of $x$, we have $\beta=\alpha$ or $(\beta, \alpha)\in E$.
\end{defn}

We remark that in Definitions \ref{defn of i-absorbed} and \ref{defn of i-good}, $\alpha$ may not be in the support of $x$.  If for any $\beta$ in the support of $x$, we have $\beta=\alpha$ or $(\beta, \alpha)\in E$, then certainly $x$ is $\alpha$-good, but the converse need not be true. If $[x]=[y]$, or $x,y$ are reduced and $s(x)=s(y)$, then $x$ is $\alpha$-good (resp. $\alpha$-absorbing) if and only if $y$ is $\alpha$-good (resp. $\alpha$-absorbing). Further, as $s(\epsilon)=\emptyset$, we have that $\epsilon$ is both $\alpha$-good and $\alpha$-absorbing, and hence so is $1_\beta$ for all $\beta\in V$.  Finally, if $w\in X^+$ and $s(w)=\{ \alpha\}$ for some $\alpha\in V$, then $w$ is $\alpha$-good.
By Remark \ref{observations for Fountainicity} we have:

\begin{lemma} \label{observations for Fountainicity-1} Let $x=x_1\circ \cdots \circ x_n, y=y_1\circ \cdots \circ y_n\in X^*$ be such that $x_i, y_i\not \in I$ and $s(x_i)=s(y_i)$ for all $1\leq i\leq n$. If one of $x,x^r,y,y^r$ is a reduced word that is $\alpha$-good, then so are all four.
\end{lemma}

The next lemma is crucial in allowing us to deduce the $\alpha$-goodness (or otherwise) of a word in terms of its factors.

\begin{lemma}\label{i-good}
Let $\alpha\in V$ and $x=x_1 \circ \cdots \circ x_n\in X^*$.

{\rm(i)} If $x_k\circ \cdots \circ x_n$ is $\alpha$-good for some $1\leq k\leq n$, then $x_1\circ \cdots \circ x_n$ is $\alpha$-good if and only if $x_1\circ \cdots \circ x_{k-1}$ is $\alpha$-good.

{\rm(ii)} If $x_1\circ \cdots \circ x_{k-1}$ is $\alpha$-good for some $1\leq k\leq n+1$,  then $x_1\circ \cdots \circ x_n$ is $\alpha$-good if and only if $x_k\circ \cdots \circ x_{n}$ is $\alpha$-good.
\end{lemma}
\begin{proof}
Suppose that $x_k\circ \cdots \circ x_n$ is $\alpha$-good.

If $x_1\circ \cdots \circ x_{k-1}$ is $\alpha$-good,
then from Remark~\ref{cor:concatenate} and comments above it is  clear that $x_1\circ \cdots \circ x_n$ is $\alpha$-good.

Conversely, suppose that $x_1\circ \cdots \circ x_n$ is $\alpha$-good but $x_1\circ \cdots \circ x_{k-1}$ is not $\alpha$-good. Let $u_1\circ \cdots \circ u_m$ be a reduced form of $x_1\circ \cdots \circ x_{k-1}$. Then, by Definition \ref{defn of i-good}, there exists some $1\leq t\leq m$ such that $s(u_t)=\beta$ with $\beta\neq \alpha$ and $(\beta, \alpha)\not \in E$. As $x_1\circ \cdots \circ x_n$ is $\alpha$-good, $\beta$ is not in the support of the reduced form of  $x_1\circ \cdots \circ x_n$. Let $v_1\circ \cdots \circ v_l$ be a reduced form of  $x_k\circ \cdots \circ x_n$. As  $x_k\circ \cdots \circ x_n$ is $\alpha$-good, $\beta$ is not in the support of $v_1\circ \cdots \circ v_l$. Now consider the word
$(u_1\circ \cdots \circ u_m)\circ (v_1\circ \cdots \circ v_l)$. Of course, $$[(u_1\circ \cdots \circ u_m)\circ (v_1\circ \cdots \circ v_l)]=[x_1 \circ \cdots \circ x_n].$$ By the dual of Corollary \ref{product of reduced}, $\beta$ lies in the  support of the reduced form of $(u_1\circ \cdots \circ u_t)\circ (v_1\circ \cdots \circ v_l)$, and hence that of  $x_1 \circ \cdots \circ x_n$, contradiction.

The proof of (ii) is the dual of (i).
\end{proof}

\begin{Cor}\label{connected i-good} Let $x\in X^*$ and let $z,z',t\in X$ where
$s(z)=s(z')=\alpha, s(t)=\beta$ and $(\alpha,\beta)\in E$. The the following are equivalent:
\begin{enumerate} \item $x$ is $\alpha$-good;
\item $z\circ z'\circ x$ is $\alpha$-good;
\item $z'\circ x$ is $\alpha$-good;
\item $zz'\circ x$ is $\alpha$-good;
\item $z\circ t\circ x$ is $\alpha$-good.
\end{enumerate}
\end{Cor}
\begin{proof}
From the remarks following Definition~\ref{defn of i-good}, $z,z\circ z'$, $zz'$ and $z\circ t$ are  $\alpha$-good. The result follows by Lemma~\ref{i-good}.
\end{proof}

 Our next definition is more subtle, but crucial for subsequent analysis of products in $\mathscr{GP}$.

\begin{defn}\label{def of condition p} Let  $\alpha\in V$ and
$x=x_1\circ \cdots \circ x_n\in X^*$ be $\alpha$-good. Then $x$  is said to be {\it $\alpha$-amenable}  if one of the following holds:

(i) $n\leq 2$;

(ii) $n>2$  and  either $\alpha\not \in s(x_2\circ \cdots \circ x_{n-1})$, or $\alpha\in s(x_2\circ \cdots \circ x_{n-1})$ and for all $x_k$ with $2\leq k\leq n-1$ such that $s(x_k)=\alpha$, the word $x_k\circ \cdots \circ x_n$ is not $\alpha$-good.
\end{defn}

 It might help to bear in mind that $x_k\circ \cdots \circ x_n$ is not $\alpha$-good if and only if there exists some $\beta\neq\alpha$ in the support of a reduced form, such that $(\alpha,\beta)\notin E$. Notice  that $\epsilon$ is $\alpha$-amenable for any $\alpha\in V$.

As we remarked earlier, for $x,y\in X^*$,  if $[x]=[y]$, then $x$ is $\alpha$-good if and only if $y$ is $\alpha$-good. One might  ask: Is it always true that $x$ is $\alpha$-amenable if and only if $y$ is $\alpha$-amenable? The answer is no, as illustrated by the following easy example. Let
$\alpha,\beta,\gamma$ be distinct elements of $ V$ with $(\alpha,\beta),(\alpha, \gamma)\in E$ and $a\in M_\alpha,
b\in M_\beta$ and $c\in M_\gamma$ non-identity elements.
The word $a\circ b\circ c$ is reduced, $\alpha$-amenable
(by virtue of $\alpha\notin s(b)$). On the other hand it shuffles to $b\circ a\circ c$
which is $\alpha$-good but not $\alpha$-amenable (as $s(a)=\alpha$ and $(\alpha,\gamma)\in E$).

On the positive side,  we have the following result.

\begin{lemma}\label{starter}
Let $\alpha\in V$ and $x=x_1\circ \cdots \circ x_n\in X^*$ be $\alpha$-amenable. Let $y$
be any word obtained by applying reduction steps and shuffles to $x_2\circ \cdots \circ x_{n-1}$. Then $x_1 \circ y\circ x_n$ is also $\alpha$-amenable.
\end{lemma}

\begin{proof} Clearly the result is true for $n\leq 2$ as here $x_2\circ \cdots \circ x_{n-1}=\epsilon $
and there are no steps to apply.

Assume now that $n>2$. Since $x$ is $\alpha$-good, so is any word   in the same equivalence class, so that   $x_1 \circ y\circ x_n$ is also $\alpha$-good. To show $x_1 \circ y\circ x_n$ is $\alpha$-amenable, it is sufficient to consider the case where $y$ is obtained from $p=x_2\circ \cdots \circ x_{n-1}$
in a single step.

Clearly, if $\alpha\notin s(p)$, then we are done; suppose therefore that $\alpha\in s(p)$.
We consider the following cases.

Case (1): $s(x_j)=\beta$ and $s(x_{j+1})=\gamma$ with $(\beta,\gamma)\in E$, where $2\leq j<n-1$.  We show that the word $$x'=x_1\circ x_2\circ \cdots \circ x_{j-1}\circ x_{j+1}\circ x_{j}\circ x_{j+2}\circ \cdots \circ x_{n-1}\circ x_n$$ is $\alpha$-amenable. Clearly, we are fine in the case where neither $\beta$ nor $\gamma$ equals $\alpha$.
If $\beta=\alpha$ (and so $\gamma\neq \alpha$), then, by Definition \ref{def of condition p}, $x_j\circ x_{j+1}\circ x_{j+2}\circ \cdots  \circ x_{n-1}\circ x_{n}$ is not $\alpha$-good. But, on the other hand, as $x_j\circ x_{j+1}$ is $\alpha$-good,  $x_{j+2}\circ \cdots \circ x_{n}$ is not $\alpha$-good by Lemma \ref{i-good}, and hence, again by Lemma \ref{i-good}, $x_{j}\circ x_{j+2}\circ \cdots \circ x_{n}$ is not $\alpha$-good.
For any $k$ with $2\leq k\leq n-1$ and $k\neq j,j+1$ with  $s(x_k)=\alpha$, it is clear that the factor $x_k\circ \cdots \circ x_n$ of $x'$ is not $\alpha$-good by the assumption that $x$ is $\alpha$-amenable.
Similarly if $\gamma=\alpha$.

Case (2): $s(x_j)=s(x_{j+1})=\beta$ where $2\leq j<n-1$. We show that the word
$$x''=x_1\circ x_2\circ \cdots \circ x_{j-1}\circ x_{j}x_{j+1}\circ x_{j+2}\circ  \cdots \circ x_{n-1}\circ x_n$$ is $\alpha$-amenable.  As in Case (1) it is enough to show that if
$\beta=\alpha$ then $x_jx_{j+1}\circ x_{j+2}\circ \cdots \circ x_n$ is not $\alpha$-good.
To this end, if $\beta=\alpha$, then as $x_{j}\circ x_{j+1}$ and $x_{j} x_{j+1}$ are  $\alpha$-good but $x_j\circ x_{j+1}\circ x_{j+2}\circ \cdots \circ x_n$ is not $\alpha$-good, we deduce from
Corollary~\ref{connected i-good} that $x_jx_{j+1}\circ x_{j+2}\circ \cdots \circ x_n$ is not $\alpha$-good.

Case (3): $s(x_j)=\beta$ and  $x_j=1_\beta$. An essentially vacuous argument easily gives  that  the word $$x'''=x_1\circ x_2\circ \cdots \circ x_{j-1}\circ x_{j+1}\circ x_{j+2} \cdots \circ x_{n-1}\circ x_n$$ is $\alpha$-amenable.
\end{proof}

The next corollary is immediate from Lemmas~\ref{lem:reduce} and ~\ref{starter}.

\begin{Cor}\label{cor:starter}
Let $\alpha\in V$ and $x=x_1\circ \cdots \circ x_n\in X^*$ be  $\alpha$-amenable. Let $y$
be a reduced form of  $x_2\circ \cdots \circ x_{n-1}$. Then $x_1 \circ y\circ x_n$ is also $\alpha$-amenable.
\end{Cor}

\begin{lemma}\label{haha}
Let $x=x_1\circ \cdots \circ x_m, y=y_1\circ \cdots y_n\in X^*$ be reduced words. If $s(x_m)=\alpha$ but $\alpha\not \in s(y)$ and there exists $\beta\in s(y)$ with $(\beta, \alpha)\not \in E$, then $\beta$ must be in the support of the reduced form of $x\circ y$.
\end{lemma}

\begin{proof}
We proceed by induction on $n$. If $n=1$, then $x\circ y=x_1\circ \cdots \circ x_m\circ y_1$.  We must have $s(y_1)=\beta$ so that $x\circ y$ is clearly reduced by Remark~\ref{cor:concatenate}.
Suppose now that $n>1$ and the result is true for all words $y$ of length strictly less than $n$.

Clearly, the result is true if $(x_1\circ \cdots \circ x_m)\circ (y_1\circ \cdots y_n)$ is reduced. If not, by Remark~\ref{cor:concatenate}, there exists some $1\leq k\leq m, 1\leq j\leq n $ such that $s(x_k)=s(y_j)$ and $(s(y_j), s(z))\in E$
for any  $z=x_h$ or $z=y_t$ with $k+1\leq h\leq m, 1\leq t\leq j-1$.  Let $y'=y_1\circ \cdots \circ y_{j-1}\circ y_{j+1}\circ \cdots \circ y_n$; notice that $y$ shuffles to
$y_j\circ y'$, so that  $y'$ is a reduced form. Further, let $p=x_1\circ \cdots \circ x_{k-1}\circ x_ky_j\circ x_{k+1}\circ\cdots \circ x_m$. Let  $x'=p$
if $x_ky_j$ is not an identity
and otherwise let  $x'=x_1\circ \cdots \circ x_{k-1}\circ x_{k+1}\circ\cdots \circ x_m$; in either case,
$x'$ is a reduced form.   Now consider $x'\circ y'$. Clearly, $[x\circ y]=[x'\circ y']$. As $\alpha\not \in s(y)$, we have $\alpha\notin s(y')$ and $s(x_k)=s(y_j)\neq \alpha$,  so that $k\neq m$. Moreover, as $s(x_m)=\alpha$ and $(\beta, \alpha)\not \in E$,  we have $s(y_j)\neq \beta$, and so $\beta\in s(y')$. By induction, $\beta$ is in the support of any reduced form of $x'\circ y'$, and hence in that of $x\circ y$.
\end{proof}

\begin{lemma}\label{absorbed-1}
Let $\alpha\in V$ and $x=x_1\circ \cdots \circ x_n\in X^*$ be $\alpha$-amenable with $s(x_n)\neq \alpha$. Then the word $x'=x_2\circ \cdots \circ x_{n-1}$ is $\alpha$-absorbing.
\end{lemma}

\begin{proof}  If $n\leq 2$, then we may take $x_2\circ \cdots \circ x_{n-1}$ as $\epsilon$, which is certainly  $\alpha$-absorbing.  Assume now that $n>2$.  Let $y=y_1\circ \cdots \circ y_m$ be a reduced form of $x'$. By Corollary \ref{cor:starter}, $x_1\circ y\circ x_{n}$ is $\alpha$-amenable. We claim that  $x'$ is $\alpha$-absorbing. To prove this, we assume the contrary, so that $T\neq\emptyset $ where
\[T=\{ k: 1\leq k\leq m, s(y_k)=\alpha\}.\]
Let $l$ and $l'$ be the least and greatest elements of $T$, respectively. Since $x_1\circ y \circ x_{n}$ is $\alpha$-amenable, we  have that for any $k\in T$ the word
$y_k\circ y_{k+1}\circ\cdots\circ y_m\circ x_n$ is not $\alpha$-good.
We consider the following  cases.

Case (1): {\em $x_1\circ y$ is a reduced form.} It follows that  $x_1\circ y_1\circ \cdots \circ y_{l'}$ is also a reduced form. Let $z$ be a reduced form of $y_{l'+1}\circ \cdots \circ y_m\circ x_n$. As commented, $\alpha$-amenability gives us that $y_{l'}\circ y_{l'+1}\circ \cdots \circ y_m\circ x_{n}$ is not $\alpha$-good.   We deduce $y_{l'+1}\circ \cdots \circ y_m\circ x_n$ is not $\alpha$-good by Corollary \ref{connected i-good}, hence neither is $z$.  Thus there exist $\beta\in s(z)$ such that $\beta\neq \alpha$ and $(\beta, \alpha)\not \in E$. Further,  as $s(x_n)\neq \alpha$ and by the minimality of $l'$ in $T$, we have $\alpha\not \in s(y_{l'+1}\circ \cdots \circ y_m\circ x_n)$ and so $\alpha\not \in s(z)$. By Lemma \ref{haha}, $\beta$ is in the support of the reduced form of $x_1\circ y_1\circ \cdots \circ y_{l'}\circ z$, but $(\beta, \alpha)\not \in E$, implying that $x_1\circ y_1\circ \cdots \circ y_{l'}\circ z$ is not $\alpha$-good, and hence neither is $x$, a contradiction.

Case (2): {\em $x_1\circ y$ is not a reduced form and $s(x_1)=\alpha$.}  By Remark \ref{cor:concatenate}, $(\beta, \alpha)\in E$ for all
 $\beta\in s(y_1\circ \cdots \circ y_{l-1})$, and so $$[x]=
 [x_1\circ y\circ x_n]=[x_1\circ y_1\circ \cdots \circ y_m\circ x_{n}]=[y_1\circ \cdots \circ y_{l-1}\circ x_1y_{l}\circ y_{l+1}\circ \cdots \circ y_m\circ x_n].$$ Notice that $y_1\circ \cdots \circ y_{l-1}\circ x_1y_{l}$ is $\alpha$-good. By $\alpha$-amenability $y_{l}\circ y_{l+1}\circ  \cdots \circ y_m\circ x_n$ is not $\alpha$-good. As  $s(y_l)=\alpha$, we deduce that $y_{l+1}\circ  \cdots \circ y_m\circ x_n$ is not $\alpha$-good by Corollary \ref{connected i-good}, so that $y_1\circ \cdots y_{l-1}\circ x_1y_{l}\circ y_{l+1}\circ \cdots \circ y_m\circ x_n$ is not $\alpha$-good by  Lemma \ref{i-good}, a contradiction.

Case (3): {\em $x_1\circ y$ is not a reduced form and $s(x_1)\neq \alpha$.}  Then, by Remark \ref{cor:concatenate}, there exists some $1\leq j\leq m$, $j\notin T$,  such that $s(y_j)=s(x_1)$ and $(s(x_1), s(y_{k}))\in E$ for all $1\leq k\leq j-1$.  We consider two sub-cases.

\ \ Case (3)(a): $j<l'$. Let $w=y_1\circ \cdots \circ y_{j-1} \circ x_1y_j\circ y_{j+1}\circ \cdots \circ y_{l'}$. Let $w'=w$   if $x_1y_j$ is not an identity, and otherwise let  $w'=y_1\circ \cdots \circ y_{j-1}\circ y_{j+1}\circ \cdots \circ y_{l'}$, so that $w'$ is a reduced form of $w$.  Let $z$ be a reduced form of $y_{l'+1}\circ \cdots \circ y_m\circ x_n$.
Then $[w'\circ z]=[x]$.   Since $y_{l'}\circ y_{l'+1}\circ \cdots \circ y_m\circ x_{n}$ is not $\alpha$-good, we deduce $y_{l'+1}\circ \cdots \circ y_m\circ x_n$ is not $\alpha$-good by Corollary \ref{connected i-good}, so that neither is $z$. Hence there exists $\beta\in s(z)$ such that
$\beta\neq \alpha$ and  $(\beta, \alpha)\not \in E$. Further, as $s(x_n)\neq \alpha$, we have $\alpha\not \in s(y_{l'+1}\circ \cdots \circ y_m\circ x_n)$ and so $\alpha\not \in s(z)$. It then follows from Lemma \ref{haha} that $\beta$ is in the support of the reduced form of $w'\circ z$.  But, $(\beta, \alpha)\not \in E$, implying that  $w'\circ z$
and hence $x$ is not $\alpha$-good,  a contradiction.

\ \ Case (3)(b): $j>l'$. Notice  first that
$[y_{l'+1}\circ \cdots \circ y_{j-1}\circ x_1y_j\circ y_{j+1}\circ \cdots \circ y_m\circ x_n]=[w]$ where $w=x_1y_j\circ y_{l'+1}\circ\cdots\circ y_{j-1}\circ y_{j+1}\circ \cdots \circ y_m\circ x_n$. We claim that $w$ is not $\alpha$-good. As $s(y_j)=s(x_1)$ and $(s(x_1), s(y_{l'}))\in E$, we have $(s(y_j), \alpha)\in E$, so that $y_j$ is $\alpha$-good.
 By   $\alpha$-amenability, $y_{l'}\circ y_{l'+1}\circ \cdots \circ y_m\circ x_n$ is not $\alpha$-good and so  $y_{l'+1}\circ \cdots \circ y_m\circ x_n$ is not $\alpha$-good by Corollary~\ref{connected i-good}. As $[y_{l'+1}\circ \cdots \circ y_{j-1}\circ y_{j}\circ y_{j+1}\circ \cdots \circ y_m\circ x_n]=[w']$ where
 $w'=y_j\circ y_{l'+1}\circ \cdots \circ y_{j-1}\circ y_{j+1}\circ \cdots \circ y_m\circ x_n$ we deduce that $w'$ is  not $\alpha$-good and so  $y_{l'+1}\circ \cdots \circ y_{j-1}\circ y_{j+1}\circ \cdots \circ y_m\circ x_n$ is not $\alpha$-good by Lemma \ref{i-good}; similarly,  as $x_1y_j$ is $\alpha$-good, we deduce $x_1 y_j\circ y_{l'+1}\circ \cdots \circ y_{j-1}\circ y_{j+1}\circ \cdots y_m\circ x_n$ is not $\alpha$-good. Let $z$ be a reduced form of  $x_1 y_j\circ y_{l'+1}\circ \cdots \circ y_{j-1}\circ y_{j+1}\circ \cdots y_m\circ x_n$ and notice $\alpha\notin s(z)$. As $z$ is not $\alpha$-good, there is $\beta\in s(z)$ such that
$\beta\neq \alpha$ and $(\beta, \alpha)\not \in E$.  Consider the word $v=y_1\circ \cdots \circ y_{l'}\circ z$. Clearly $[
x]=[v]$.  By Lemma \ref{haha}, $\beta$ is in the support of the reduced form of $v$ and hence that of $x$. But $(\beta, \alpha)\not \in E$, contradicting $x$  being $\alpha$-good.

We conclude that $x_2\circ\cdots\circ x_{n-1}$ is $\alpha$-absorbing, thus completing the proof.
\end{proof}

\begin{Cor}\label{transfer}
Let $\alpha\in V$ and $x=x_1\circ \cdots \circ x_n\in X^*$ be $\alpha$-amenable with $s(x_n)=\alpha$, and let   $\beta\in V$.
Then   $x_1\circ \cdots \circ x_{n-1}\circ 1_\beta$  is also $\alpha$-amenable.
\end{Cor}

\begin{proof}
Certainly $1_\beta$ is  $\alpha$-good, as its unique reduced form is $\epsilon$.
Since $s(x_n)=\alpha$ and  $x$ is $\alpha$-good, two applications of Lemma \ref{i-good}
give that  $x_1\circ \cdots \circ x_{n-1}\circ 1_\beta$ is $\alpha$-good.
Suppose that $n\geq 3$ and $s(x_k)=\alpha$ where $2\leq k\leq n-1$. By $\alpha$-amenability, $x_k\circ \cdots \circ x_n$ is not $\alpha$-good, but as $x_n$ is $\alpha$-good,
two applications of Lemma \ref{i-good} give that  $x_k\circ \cdots \circ x_{n-1}\circ 1_\beta$ is not $\alpha$-good. Therefore  $x_1\circ \cdots \circ x_{n-1}\circ 1_\beta$ is $\alpha$-amenable.
\end{proof}

We have been working towards the following:

\begin{Prop}\label{absorbed}
Let $\alpha\in V$ and $x=x_1\circ \cdots \circ x_n\in X^*$ be $\alpha$-amenable. Then the factor $x_2\circ \cdots \circ x_{n-1}$ is $\alpha$-absorbing.
\end{Prop}

\begin{proof}
The result is true when $s(x_n)\neq \alpha$, by  Lemma \ref{absorbed-1}. Suppose that $s(x_n)= \alpha$. By Corollary~ \ref{transfer}, $x_1\circ \cdots \circ x_{n-1}\circ 1_\beta$ is $\alpha$-amenable, for any $\beta\in V$. Since $|V|\geq 2$, taking  $\beta\neq \alpha$ Lemma \ref{absorbed-1} tells us  that $x_2\circ \cdots \circ x_{n-1}$ is $\alpha$-absorbing.
\end{proof}

\begin{Cor}\label{move}
Let $\alpha\in V$ and $x=x_1\circ \cdots \circ x_n\in X^*$  be $\alpha$-amenable.

{\rm(i)} If $s(x_1)=s(x_n)=\alpha$, then for all $\beta$ in the support of the reduced form of $x_2\circ \cdots \circ x_{n-1}$ we have $\beta\neq \alpha$ and $(\alpha, \beta)\in E$.

{\rm(ii)} If $s(x_1)=\alpha, s(x_n)\neq \alpha$, then for all $\beta$ in the support of the reduced form of $x_2\circ \cdots \circ x_{n}$ we have $\beta\neq \alpha$ and $(\alpha, \beta)\in E$.
\end{Cor}

\begin{proof}  Clearly we may assume that $n>2$. Let $y$ be a reduced form of $x_2\circ \cdots \circ x_{n-1}$, so that  $\alpha\notin s(y)$ by Proposition \ref{absorbed}.

(i)    Clearly, the result is true when $y=\epsilon$, so we assume that $y\neq \epsilon$. Let $w$ be a reduced form of $x_1\circ y$. It  follows from Corollary \ref{product of reduced} that $s(y)\subseteq s(w)$. Further, by the dual of Corollary  \ref{product of reduced}, $s(y)$ is contained in  the support of the reduced form of $w\circ x_n$. As $x$ is $\alpha$-good, so are $x_1\circ y\circ x_n$ and $w\circ x_n$, implying $(\alpha,\beta)\in E$ for all $\beta\in s(y)$.

(ii)  Let $w$ be a reduced form such that
$[w]=[y\circ x_n]=[x_2\circ\cdots\circ x_n]$. Since $s(x_n)\neq \alpha$, we deduce that $\alpha\notin s(w)$. Let $v$ be a reduced form of  $x_1\circ w$.  Since $x$ is $\alpha$-good and $[v]=[x_1\circ w]=[
x]$, we have that $v$ is $\alpha$-good, so that $\beta=\alpha$ or $(\beta, \alpha)\in E$ for all $\beta\in s(v)$. Further, as $s(x_1)=\alpha$ but $\alpha\not \in s(w)$, we  have $s(w)\subseteq s(v)$  by Corollary \ref{product of reduced}, so that  $(\beta, \alpha)\in E$ for all $\beta\in s(w)$.
\end{proof}

In what follows we use  the foregoing analysis to allow us to factorise elements of $\mathscr{GP}$ in a way that will enable us to achieve the aim of this section. First, another definition.

\begin{defn}\label{def of N}
Let $x=x_1\circ \cdots \circ x_n\in X^*$  and $\alpha\in V$. We define a set $$N_\alpha(x)=\{k\in \{1, \cdots, n\}: s(x_k)=\alpha\mbox{ and }  x_k\circ \cdots \circ x_n \mbox{~is~} \alpha\mbox{-good}\}.$$

We will show that for a word $x$ as in Lemma~\ref{def of N} we can move the letters indexed by elements of $N_{\alpha}(x)$ to the right of $x$ (maintaining their order).  Where convenient, in situations where the enumeration of indices is particularly involved, and where there is no danger of ambiguity, we may identify $N_\alpha(x)$ with $\{ x_k:k\in N_\alpha(x)\}$.
\end{defn}

Notice that $N_\alpha(x)$ may be empty and,
in particular, $N_\alpha(\epsilon)=\emptyset$. Further,   $s(x_n)=\alpha$ if and only if $n\in N_\alpha(x)$. If $l, k\in N_\alpha(x)$ with $l<k$, there may exist some $l<j<k$ with $s(x_j)=\alpha$ such that $j\not \in N_\alpha(x)$. For example, suppose that  $n=6$,  $s(x_1)=s(x_3)=s(x_4)=s(x_6)=\alpha$, and $s(x_2)=s(x_5)=\beta$ where  $\alpha\neq \beta$, $(\alpha, \beta)\not \in E$, $x_3x_4=1_\alpha$, $x_2,x_5\notin I$ and $x_2x_5=1_\beta$. Then $N_\alpha(x)=\{1, 6\}$. This also provides an example of an $\alpha$-amenable word.

\begin{lemma}\label{pre}
Let $\alpha\in V$ and $x=x_1\circ \cdots \circ x_n\in X^*$ with $s(x_n)=\alpha$. Write $$N_\alpha(x)=\{l_1, \cdots, l_r: 1\leq l_1< \cdots < l_r= n\}.$$ Then $$[x]=[x'] [x_{l_1}\circ \cdots \circ x_{l_r}]$$
where $x'$ is the word obtained from $x_1\circ \cdots \circ x_n$   by deleting the letters $x_{l_1}, \cdots, x_{l_r}$. 

Further, if $z$ is a word obtained from $x$ by replacing $x_{l_1}, \cdots, x_{l_r}$
by   letters $z_{l_1}, \cdots, z_{l_r}\in M_\alpha$, respectively, we have \[[z]=[x'][z_{l_1}\circ \cdots \circ z_{l_r}].\]
\end{lemma}

\begin{proof} Let $1\leq k\leq r-1$.

 Definition \ref{def of N}, and two applications of
 Lemma~\ref{i-good} give $x_{l_k}\circ \cdots \circ x_{l_{k+1}}$ is $\alpha$-good.
We now claim that   $x_{l_k}\circ \cdots \circ x_{l_{k+1}}$
is $\alpha$-amenable.

Clearly, $x_{l_k}\circ \cdots \circ x_{l_{k+1}}$ is $\alpha$-amenable if either $l_{k+1}=l_k+1$ or $l_{k+1}>l_k+1$ and  there exists no $l_k<j< l_{k+1}$ such that $s(x_j)=\alpha$. Suppose now that there exists  $l_k<j< l_{k+1}$ such that $s(x_j)=\alpha$.  Since $j\not \in N_\alpha(x)$, the word $x_{j}\circ \cdots \circ x_{n}$ is not $\alpha$-good. On the other hand, we know $x_{l_{k+1}+1}\circ \cdots \circ x_{n}$ is $\alpha$-good,  giving that $x_j\circ \cdots \circ x_{l_{k+1}}$ is not $\alpha$-good by Lemma \ref{i-good}, and hence $x_{l_k}\circ \cdots \circ x_{l_{k+1}}$ is $\alpha$-amenable.

For each  $k$ in the range above let $w_k$ be a reduced form of $x_{l_k+1}\circ \cdots \circ x_{l_{k+1}-1}$.  
By Corollary \ref{move}, since
$x_{l_k}\circ \cdots \circ x_{l_{k+1}}$
 is $\alpha$-amenable, for any $\beta\in s(w_k)$ we have
 $\beta\neq \alpha$ and $(\beta, \alpha)\in E$. Further,
 $$[x]=[x_1\circ \cdots \circ x_{l_r}]=[x_1\circ \cdots \circ x_{l_1}\circ w_1\circ x_{l_2}\circ \cdots \circ x_{l_{r-1}}\circ w_{r-1}\circ x_{l_r}],$$ so that
$$[x]=[y'\circ x_{l_1}\circ \cdots \circ x_{l_r}]=
[y'][ x_{l_1}\circ \cdots \circ x_{l_r}]=
[x'][x_{l_1}\circ \cdots \circ x_{l_r}]$$
where $y'$ is the word obtained form $x_1\circ \cdots \circ x_{l_1}\circ w_1\circ x_{l_2}\circ \cdots \circ x_{l_{r-1}}\circ w_{r-1}\circ x_{l_r}$ by deleting $x_{l_1}, \cdots, x_{l_r}$ and $x'$ is the word obtained from $x_1\circ \cdots \circ x_n$ by deleting $x_{l_1}, \cdots, x_{l_r}$.

Suppose now that $z$ is a word obtained from $x$ by replacing $x_{l_1}, \cdots, x_{l_r}$
by  letters $z_{l_1}, \cdots, z_{l_r}\in M_\alpha$, respectively.  Since $[y']=[x']$, we have
$$[z]= [y'\circ z_{l_1}\circ \cdots\circ z_{l_r}]=
[y'][z_{l_1}\circ \cdots\circ z_{l_r}]=[x'][ z_{l_1}\circ \cdots\circ z_{l_r}].$$
\end{proof}

 We  now remove the restriction that $s(x_n)=\alpha$ in Lemma~\ref{pre}.

\begin{lemma}\label{product}
Let  $\alpha\in V$ and $x=x_1\circ \cdots \circ x_n\in X^*$. Write $$N_\alpha(x)=\{l_1, \cdots, l_r: 1\leq l_1< \cdots < l_r\leq n\}.$$ Then $$[x]=[x'] [x_{l_1}\circ \cdots \circ x_{l_r}]$$
where $x'$ is the word obtained from  $x_1\circ \cdots \circ x_n$ by deleting the letters $x_{l_1}, \cdots, x_{l_r}$.

Further, if $z$ is a word obtained from $x$ by replacing $x_{l_1}, \cdots, x_{l_r}$
 by  letters $z_{l_1}, \cdots, z_{l_r}\in M_\alpha$, respectively, we have
 \[[z]=[x'][z_{l_1}\circ \cdots \circ z_{l_r}].\]
\end{lemma}

\begin{proof} We are done with the case where $s(x_n)=\alpha$, by Lemma \ref{pre}. Suppose now that $s(x_n)\neq \alpha$, and so $l_r\neq n$.
Let $p=x_1\circ \cdots \circ x_{l_r}$.
 Applications of  Lemma~\ref{i-good} that are now standard yield  $N_\alpha(p)=\{l_1, \cdots, l_r\}$.
By Lemma \ref{pre}, $[p]=[p'\circ x_{l_1}\circ \cdots \circ x_{l_r}]$ where $p'$ is the word obtained from $p$ by deleting letters $x_{l_1}, \cdots, x_{l_r}$. 
 We now have $$[x]=[p\circ x_{l_r+1}\circ \cdots \circ x_n]=[p'\circ x_{l_1}\circ \cdots \circ x_{l_r}\circ x_{l_r+1}\circ \cdots \circ x_n].$$

To show the required result, we now consider the $\alpha$-good word $x_{l_r}\circ \cdots \circ x_n$. We now claim that it is $\alpha$-amenable. Clearly,  we are done with the cases where either $n=l_r+1$ or $n>l_r+1$ and there exists no $l_r<j<n$ such that $s(x_j)=\alpha$. Suppose therefore that there exists $l_r<j<n$ such that $s(x_j)=\alpha$. As $j\not \in N_\alpha(x)$, we have that $x_j\circ \cdots \circ x_n$ is not $\alpha$-good, and so  $x_{l_r}\circ \cdots \circ x_n$ is $\alpha$-amenable. Let $q$ be a reduced form of $x_{l_r+1}\circ \cdots \circ x_n$. Since $x_{l_r}\circ \cdots \circ x_n$ is $\alpha$-amenable and $s(x_n)\neq \alpha$, we have that  $\beta\neq\alpha$ and $(\alpha, \beta)\in E$ for all  $\beta\in s(q)$ by Corollary \ref{move}. Therefore, $$[x]=[p'\circ x_{l_1}\circ \cdots \circ x_{l_r}\circ x_{l_r+1}\circ \cdots \circ x_n]=[p'\circ x_{l_1}\circ \cdots \circ x_{l_r}\circ q]=[p'\circ q \circ x_{l_1}\circ \cdots \circ x_{l_r}].$$  Since $[p'\circ q]=[x']$, we have $$[x]=[x'][ x_{l_1}\circ \cdots \circ x_{l_r}].$$

Suppose now that $z$ is a word obtained from $x$ by replacing $x_{l_1}, \cdots, x_{l_r}$
by  letters $z_{l_1}, \cdots, z_{l_r}$ from $M_\alpha$, respectively. Clearly,
$z=z'\circ x_{l_r+1}\circ \cdots \circ x_n $ where $z'$ is the word obtained from
$p$  by replacing $x_{l_1}, \cdots, x_{l_r}$
by  $z_{l_1}, \cdots, z_{l_r}\in M_\alpha$.
We have  shown that $N_\alpha(p)=\{l_1, \cdots, l_r\}$ and so from   Lemma~\ref{pre}
we have
$[z']=[p'\circ z_{l_1}\circ \cdots \circ z_{l_r}].$
Then
$$[z]=[z'\circ x_{l_r+1}\circ \cdots \circ x_n]=[p'\circ z_{l_1}\circ \cdots \circ z_{l_r}\circ q]=[p'\circ q\circ z_{l_1}\circ \cdots \circ z_{l_r}]=[x'][z_{l_1}\circ \cdots\circ z_{l_r}].$$

\end{proof}

 The reader should note that we are {\em not} claiming that the maps
$\overline{\phi}_\alpha$ and
$\overline{\psi}_\alpha$ in Lemma~\ref{phi} are morphisms.

\begin{Lem}\label{phi}
Let $\alpha\in V$. Then  the maps
$$\phi_\alpha: X^*\longrightarrow \mathscr{GP}\mbox{ and } \psi_\alpha: X^*\longrightarrow \mathscr{GP}$$ defined by  $$x\phi_\alpha=[x_{l_1}\circ \cdots \circ  x_{l_r}]\mbox{ and } x\psi_\alpha=[x_{m_1}\circ \cdots \circ x_{m_t}]$$ where  $x=x_1\circ \cdots \circ x_n$, with
\[N_\alpha(x)=\{l_1, \cdots, l_r\}, \,\,  1\leq l_1<\cdots <l_r\leq n\]
and
\[\{m_1, \cdots, m_t\}=\{1, \cdots, n\} \backslash N_\alpha(x)
,\,\,  1\leq m_1<\cdots <m_t\leq n,\]  induce maps  $$\overline{\phi}_\alpha:  \mathscr{GP}\longrightarrow \mathscr{GP}\mbox{ and } \overline{\psi}_\alpha: \mathscr{GP}\longrightarrow \mathscr{GP}$$ defined by $$[x]\overline{\phi}_\alpha=x\phi_\alpha\mbox{ and } [x]\overline{\psi}_\alpha=x\psi_\alpha.$$ Further,
$[x]=(x\psi_\alpha)(x\phi_\alpha)$.
\end{Lem}
\begin{proof}
To show that $\overline{\phi}_\alpha$ and $\overline{\psi}_\alpha$ are well defined we need to show that  $R^\sharp\subseteq \ker \phi_\alpha$ and $R^\sharp\subseteq \ker \psi_\alpha$. Let $L$ be the binary relation on $X^*$ defined by $$L=\{(y\circ a\circ z,y\circ b\circ z): y,z\in X^*, (a, b)\in R\}.$$  Since $R^\sharp$ is the transitive closure of $L$, and $\ker \phi_\alpha$ and $\ker \phi_\alpha$ are, of course, equivalence relations, it suffices to show that $L\subseteq \ker \phi_\alpha$ and $L\subseteq \ker \psi_\alpha$. This can be seen in a routine manner by using Corollary~\ref{connected i-good} and considering $(a,b)\in R_{id}, R_{v}$ and $R_e$ in turn.

It follows from Lemma \ref{product} that $[x]=(x\psi_\alpha)(x\phi_\alpha)$.
\end{proof}

\begin{Prop}\label{step-2}
 Let $z=z_1 \circ \cdots \circ z_n\in X^*$
 such that $s(z)$ is a complete subgraph  such that $s(z_j)\neq s(z_k)$ for any
 $1\leq j<k\leq n$.  Suppose that $z_k\,\mathcal{R}^*\, z_k'$ in $M_{s(z_k)}$
 for $1\leq k\leq n$ and put $z'=z_1' \circ \cdots \circ z_n'.$ 
 Then $[z]\,\mathcal{R}^*\, [z']$ in  $\mathscr{GP}$.
\end{Prop}
\begin{proof}
Let $x=x_1\circ \cdots \circ x_m,\,  y=y_1\circ \cdots \circ y_h\in X^*$ be such that $[x][z]=[y][z]$. We proceed by induction on  $n$ to show $[x][z']=[y][z']$. Clearly, the result is true when $n=|z|=0$, i.e. $z =\epsilon=z'$. Suppose now that $n>0$ and the result is true for all such $z$ with $|z|<n$.
 Let $s(z_1)=\alpha$.  Then
$s(z_k)\neq \alpha$ and $(\alpha,s(z_k))\in E$ for all $1<k\leq n$, so that certainly $z$ is $\alpha$-good. Suppose that
$$N_\alpha(x\circ z)=\{r_1, \cdots, r_l\}\mbox{ and } N_\alpha(y\circ z)=\{d_1, \cdots, d_{t} \}$$
 where
$$r_1< \cdots< r_l\mbox{ and } d_1< \cdots< d_{t}.$$ Since $z=z_1\circ \cdots\circ z_n$ is a complete block and $s(z_1)=\alpha$, we have that $z_1$ is the last letter in $x\circ z$ with support $\alpha$ and $z$ is clearly $\alpha$-good, so that $r_l=m+1$ by Definition \ref{def of N}. Similarly, $d_t=h+1$. By Lemma \ref{product},
$$[x\circ z]=[x'\circ z_2\circ \cdots \circ z_n][x_{r_1}\circ \cdots  \circ x_{r_{l-1}}\circ z_1]$$ and $$[y\circ z]=[y'\circ z_2\circ \cdots \circ z_{n}][y_{d_1}\circ \cdots \circ y_{d_{t-1}}\circ z_1].$$  By replacing the first letter $z_1$ of $z$ by $z_1'$ in $x\circ z$, we have
$$[x\circ z_1'\circ \cdots \circ z_n]=[x'\circ z_2\circ \cdots \circ z_n][x_{r_1}\circ \cdots  \circ x_{r_{l-1}}\circ z_1']$$ by Lemma \ref{product}. Similarly,
$$[y\circ z_1'\circ \cdots \circ z_n]=[y'\circ z_2\circ \cdots \circ z_{n}][y_{d_1}\circ \cdots \circ y_{d_{t-1}}\circ z_1'].$$

On the other hand, by applying the maps $\overline{\phi}_\alpha$ and $\overline{\psi}_\alpha$  to each side of $[x\circ z]=[y\circ z]$, we have $$[x'\circ z_2\circ \cdots \circ z_n]=[y'\circ z_2\circ \cdots \circ z_{n}],
\mbox{ and }[x_{r_1}\circ \cdots  \circ x_{r_{l-1}}\circ z_1]=[y_{d_1}\circ \cdots \circ y_{d_{t-1}}\circ z_1].$$
 Using Remark~\ref{rem:retract}, the latter gives  $x_{r_1}\cdots x_{r_{l-1}} z_1=y_{d_1}\cdots y_{d_{t-1}}z_1$. As $z_1~\mathcal{R}^*z_1'$
in $M_\alpha$,  we have $$x_{r_1}\cdots x_{r_{l-1}} z_1'=y_{d_1}\cdots y_{d_{t-1}}z_1'$$ so that $[x_{r_1}\circ \cdots  \circ x_{r_{l-1}}\circ z_1']=[y_{d_1}\circ \cdots \circ y_{d_{t-1}}\circ z_1']$. Therefore, $$[x\circ z_1'\circ \cdots \circ z_n]=[y\circ z_1'\circ \cdots \circ z_n]$$ and so
$$[x\circ z_1'][z_2\circ \cdots \circ z_n]=[y\circ z_1'][z_2\circ \cdots \circ z_n].$$
Our inductive assumption now gives
$$[x][ z_1'\circ z_2'\circ \cdots \circ z_n']=[x\circ z_1'][z_2'\circ \cdots \circ z_n']=[y\circ z_1'][z_2'\circ \cdots \circ z_n']=[y][ z_1'\circ z_2'\circ \cdots \circ z_n'].$$
The result follows by induction.
\end{proof}

\begin{Prop}\label{prop:mainprop}  Let $u\in X^*$ and let $[u]=[a][v]$
where $a,v\in X^*$ are such that all letters contained in $a$ are left invertible, and $v=v_1\circ \cdots \circ v_m$ is a left Foata normal form with blocks $v_k$, $1\leq k\leq n$, such that $v_1$ contains no left invertible letters.  Let
$v_1=z_1\circ \cdots \circ z_s\in X^*$. Suppose that for each $1\leq j\leq s$ an idempotent $z_j^+\in M_{s(z_j)}$ is chosen such  that
 $z_j^+~\mathcal{R}^*~z_j$
in $M_{s(z_j)}$, and put  $v_1^+=z_1^+\circ \cdots \circ z_s^+$.
 Let $[a']$ be a left inverse of $[a]$ in $\mathscr{GP}$. Then
$$[u]~\mathcal{R}^*~[a]  [v_1^+] [a']$$
and  $[a] [v_1^+] [a']$ is  idempotent.
\end{Prop}

\begin{proof} Under  the conditions of the hypothesis, it follows from (1) of Proposition~\ref{step-1} that $[v]\,\ars\, [v_1]$ and then from Proposition~\ref{step-2} that $[v_1]\,\ars\, [v_1^+]$. Since $[a'][a]=[\epsilon]$, we have $[a']\,\ar\, [\epsilon]$ and so certainly $[a']\,\ars\, [\epsilon]$. Then
$$[u]=[a] [v]~\mathcal{R}^*~[a] [v_1^+]~\mathcal{R}^*~[a]  [v_1^+] [a'],$$
using  the fact that  $\ars$ is a left congruence.
Further,  $[a] [v_1^+] [a']$ is  idempotent  by Lemma \ref{idempotent}.
\end{proof}

The main result of our paper now follows.

\begin{Thm}\label{thm:mainm}
The graph product $\mathscr{GP}=\mathscr{GP}(\Gamma,\mathcal{M})$ of left abundant monoids $\mathcal{M}=\{M_\alpha: \alpha\in V\}$ with respect to $\Gamma$ is  left abundant.
\end{Thm}
\begin{proof} Let $[u]\in \mathscr{GP}$. By  Lemma~\ref{decomposition} we are guaranteed a decomposition of $u$ as in Proposition~\ref{prop:mainprop}. The result now follows from the assumption that each
vertex monoid is left abundant.
\end{proof}

Of course, the left-right dual of Theorem ~\ref{thm:mainm} holds, and hence one may also deduce that the graph product of abundant monoids is abundant. A consequence is worth stating separately.

\begin{Cor}\label{cor:reg}
The graph product $\mathscr{GP}=\mathscr{GP}(\Gamma,\mathcal{M})$ of regular monoids $\mathcal{M}=\{M_\alpha: \alpha\in V\}$ with respect to $\Gamma$ is   abundant.
\end{Cor}

\section{Graph Products of left Fountain Monoids are left Fountain}\label{sec:Fountainicity}

We now discuss the left Fountainicity of the graph product $\mathscr{GP}=\mathscr{GP}(\Gamma,\mathcal{M})$ of left Fountain monoids $\mathcal{M}=\{M_\alpha: \alpha\in V\}$ with respect to $\Gamma$.

Our strategy is as follows. We know from
Lemma~\ref{decomposition} that
any element of $\mathscr{GP}$ has  reduced form
$a\circ x$ where the letters of $a$ are all left invertible, $x=x_1\circ \cdots \circ x_n$ is a left Foata normal form with blocks $x_i$, $1\leq i\leq n$, such that $x_1$ contains no  left invertible letters. From  Proposition~\ref{step-1} we then have
$[a\circ x]\,\ars\,  [a\circ x_1]$ and so certainly $[a\circ x]\,\art\,  [a\circ x_1]$.
We take an idempotent of  $\mathscr{GP}$ in standard form $u$ and examine the reduction processes for the word
$u\circ a\circ x_1$ in the case $[u\circ a \circ x_1]=
[a\circ x_1]$. This eventually enables us to show that  $[u\circ a \circ x_1]=
[a\circ x_1]$ if and only if  $[u\circ a \circ \overline{x}_1]=
[a\circ \overline{x}_1]$ where $\overline{x}_1$ is obtained from $x_1$ by replacing each letter by an idempotent in the same $\art$-class in the relevant vertex monoid. Hence
$[a\circ x_1]\,\art\, [a\circ \overline{x}_1]$ but then with
$[a']$ being a left inverse for $[a]$ we arrive at
$[a\circ x_1]\,\art\, [a\circ \overline{x}_1\circ a']$. The latter element is clearly idempotent.

To proceed, we rely on the
analysis of $\alpha$-good suffices of words provided in Section~\ref{sec:abundancy}. In addition, we need some further analysis of the way in which the product of two reduced words reduces in $\mathscr{GP}$.

It is worth remarking that if every vertex monoid has the property that left invertible elements are also right invertible, then our arguments would need to be less delicate. Since, in that case, $[u\circ a\circ x_1]=
[a\circ x_1]$ if and only if $[a'\circ u\circ a\circ x_1]=[x_1]$, and the fact that $s(x_1)$ is complete then makes the subsequent analysis somewhat easier.

Lemma~\ref{product reduction} shows the different ways in which multiplying a reduced word by $p\in X\setminus I$ leads to a reduced word. In some cases, we need to delete a letter of $I$, that is, use Step (id) of Definition~\ref{defn:reduction}; in other cases, we need only
Steps  (v) and (e). This leads to the following notion.

\begin{defn}\label{H}
Let $x=x_1\circ \cdots \circ x_n, \,\,y=y_1\circ \cdots \circ y_m\in X^*$ be reduced words. We say that $x\circ y$ is {\it $S$-reducible} {if in reducing $x\circ y$ to a reduced form we only use Steps (v) and (e) in  Definition~\ref{defn:reduction}.}
\end{defn}

 We use the term `$S$-reducible' since using Steps (v) and (e) would be allowed in the corresponding notion of a {\em semigroup} graph product: see Section~\ref{sec:application}.

\begin{lemma}\label{product reduction 2}
Let $x=x_1\circ \cdots \circ x_n,\,\, y=y_1\circ \cdots \circ y_m\in X^*$ be reduced words. Suppose that $x\circ y$ is $S$-reducible. Then $x\circ y$ shuffles to
\[p_1\circ\cdots \circ p_n\circ y'\]
and has reduced form
$$q_1\circ \cdots \circ q_n\circ y'$$ where for all $1\leq j\leq n$, $q_j=x_j=p_j$  or $p_j=x_j\circ y_{r_j}$ and $q_j=x_jy_{r_j}$
 for some distinct indices
$ r_j\in \{ 1,\cdots , m\}$,   and  $y'\in X^*$ is the word obtained from $y$ by deleting the letters $y_{r_j}$.
\end{lemma}

\begin{proof}
We use induction on the length $n$ of $x$. Clearly, the result is true for $n=1$ by Lemma \ref{product reduction}. Suppose that $n>1$ and the result is true for all reduced words $x$ of length strictly less than $n$.  Let $x'=x_2\circ \cdots \circ x_n$. Clearly, $x'\circ y$ is also $S$-reducible, and so $x'\circ y$
shuffles to
\[u_1=p_2\circ\cdots \circ p_n\circ y'\]
and has
a reduced form $$u_2=q_2\circ \cdots \circ q_n\circ y'$$ where for all $2\leq j\leq n$, $q_j=x_j=p_j$ or $p_j=x_j\circ y_{r_j}$ and $q_j=x_jy_{r_j}$
for some distinct indices $r_j\in \{ 1,\hdots ,m\}$,  and  $y'$ is the word obtained from $y$ by deleting the letters $y_{r_j}$.

Now consider the words $$w_1=x_1\circ u_1
\mbox{ and }w_2=x_1\circ u_2.$$
If $w_2$  is  reduced then we are done, with $p_1=q_1=x_1$. Suppose therefore that  $w_2$ is not reduced.  Since $s(q_j)=s(x_j)$ for all $2\leq j\leq n$, the word $x_1\circ q_2\circ \cdots \circ q_n$ is reduced by  Remark~\ref{observations for Fountainicity}. So, there must exist some letter $y_t$ in $y'$ with $s(x_1)=s(y_t)$ that can be shuffled to the front of
both  $u_1$ and $u_2$. Clearly $t$ is distinct from any existing $r_j$; we put $r_1=t$.  As $x\circ y$ is $S$-reducible, $x_1y_{r_1}$ is not  an identity. Therefore, $w$ shuffles  to
$$p_1\circ p_2\circ \cdots \circ p_n\circ y''$$
and, from Lemma~\ref{product reduction}, has reduced form
$$q_1\circ q_2\circ \cdots \circ q_n\circ y''$$
where
$p_1=x_1\circ y_{r_1}$ and $q_1=x_1y_{r_1}$ and $y''$ is the word obtained by deleting $y_{r_1}$ from $y'$.
\end{proof}

\begin{Cor}\label{H-2}
Let $\alpha\in V$ and let $x, y\in X^*$ be reduced words such that  $x$ is not $\alpha$-good but $x\circ y$ is $\alpha$-good. Then $x\circ y$ is not $S$-reducible.
\end{Cor}

\begin{proof} Let $x,y$ be as given.
If $x\circ y$ is $S$-reducible, then $s(x)$ is a subset of the support of the reduced form of $x\circ y$, by Lemma \ref{product reduction 2}. Since  $x$ is not $\alpha$-good,  neither is $x\circ y$,
a  contradiction.
\end{proof}

In what follows,  we use $u=b\circ e\circ b'$ to denote a standard form of an idempotent $[u]\in \mathscr{GP}$, as described in Definition \ref{standard idempotent}.  We use $a\circ x$ to denote a word in $X^*$ satisfying the following conditions:

(a) $a=a_1\circ \cdots \circ a_l$ is a reduced word such that all letters in $a$ are left invertible;

(b) $x=x_1\circ \cdots \circ x_k$ such that $s(x)$ is complete and $s(x_j)\neq s(x_t)$ for all $1\leq j<t\leq k$;

(c) there exists no $j$ with $1\leq j\leq l$ such that $(s(a_j), s(a_t))\in E$ for all $j+1\leq t\leq l$ and $s(a_j)\in s(x)$.

The reader by now might think we should assume $a\circ x$ is reduced and no letter in $x$ is left invertible. However, we need this rather looser set up. The reason for this  will become apparent later, when we apply Lemma~\ref{key-F} iteratively in
Corollary~\ref{Rtilde}.

\begin{lemma}\label{key5}
Let $a\circ x$ be defined as above. Then

{\rm (i)} for any $y=y_1\circ \cdots \circ y_k\in X^*$ such that $s(y_j)=s(x_j)$ for all $1\leq j\leq k$, $a\circ y$ is of the same form as $a\circ x$;

{\rm (ii)} $a\circ x'$ is a reduced form of $a\circ x$, where $x'$ is the word obtained from $x$ by deleting all letters in $x$ which are identities;

{\rm (iii)} for each  $\alpha\in s(x)$, $N_\alpha(a\circ x)$ contains the unique letter $x_j$ in $x$ such that $s(x_j)=\alpha$.
\end{lemma}

\begin{proof}
(i)  and (ii) are clear.

(iii) Let $\alpha\in s(x)$ and let $j$ be the unique index guaranteed by (b) such that $s(x_j)=\alpha$. Since $s(x)$ is complete, $x_j\in N_\alpha(a\circ x)$. Suppose (with some abuse of notation) that $a_h\in N_\alpha(a\circ x)$. Then $s(a_h)=\alpha$ and $a_h\circ \cdots \circ a_l\circ x$ is $\alpha$-good, hence so is its reduced form $a_h\circ \cdots \circ a_l\circ x'$. Let $h\leq t\leq l$ be the largest such that $s(a_t)=\alpha$. Then $(s(a_t), s(a_r))\in E$ for all $t+1\leq r\leq l$, contradicting  (c). Thus, $N_\alpha(a\circ x)=\{x_j\}$.
\end{proof}

 In Corollary~\ref{case1}, and Lemmas \ref{case2} and \ref{key-F}
let $a\circ x$  and
$u=b\circ e\circ b'$ be defined as above such that
$[u][a\circ x]=[a\circ x]$, and let $w=u\circ a \circ x$.

\begin{lemma}\label{lemcase1} Suppose that $u$ is $\alpha$-good.
 For any $j\in \{ 1,\cdots ,n\}$ we have
$b_j'\in N_\alpha(u)$ if and only if $b_j\in N_\alpha(u)$.
\end{lemma}
\begin{proof}  Using  Lemma  \ref{i-good}, Corollary~\ref{connected i-good} and  Lemma \ref{observations for Fountainicity-1}, the following are equivalent
\[\begin{array}{r}
b_j'\in N_\alpha(u)\\
b_j'\circ \cdots \circ b_1'\mbox{ is $\alpha$-good}\\

b_j\circ \cdots \circ b_1\mbox{ is $\alpha$-good}\\
b_1\circ \cdots \circ b_j\mbox{ is $\alpha$-good}\\
b_{j+1}\circ \cdots \circ b_n\circ e\circ b'\mbox{ is $\alpha$-good}\\
b_j\circ b_{j+1}\circ \cdots \circ b_n\circ e\circ b'\mbox{ is $\alpha$-good}\\
b_j\in N_\alpha(u).\end{array}\]
\end{proof}

We can now make progress in the case where $s(x_1)=\alpha$ and  $a\circ x$ is $\alpha$-good.

\begin{Cor}\label{case1}
 Suppose that $s(x_1)=\alpha$ and  $a\circ x$ is $\alpha$-good. Then for any $j\in \{ 1,\cdots ,n\}$ we have
$b_j'\in N_\alpha(w)$ if and only if $b_j\in N_\alpha(w)$.
\end{Cor}
\begin{proof} Since $a\circ x$ is $\alpha$-good, so is
$u\circ a\circ x$ and hence from  Lemma~\ref{i-good} so is $u$. Moreover (with substantial abuse of notation),
$z\in N_\alpha(u)$ if and only if $z\in N_\alpha(w)$, for any letter $z$ of $u$. The result follows from Lemma~\ref{lemcase1}.
\end{proof}

Without the assumption that $a\circ x$ is $\alpha$-good, our analysis of the elements of $N_\alpha(w)$ becomes more delicate. We remark that in what follows, we could replace the suffix $a\circ x$ of $w$ by any word $v$ and the same argument would apply to $u\circ v$ as it does to $w$.

\begin{lemma}\label{case2} Let $\alpha\in V$. If $b_j'\notin N_\alpha(w)$ for all $1\leq j\leq n$, then
$b_j\notin  N_\alpha(w)$ for all $1\leq j\leq n$.
\end{lemma}
\begin{proof} If $\alpha\notin s(b)$ there is nothing to show. Otherwise, let $h$ be greatest such that $s(b_h')=\alpha$, so that
\[v=b_h'\circ b_{h-1}'\circ\cdots \circ b_1'\circ a\circ x\]
is not $\alpha$-good.  Suppose that there exist some $b_j\in N_\alpha(w)$, so that
\[ z =  b_j\circ\cdots \circ b_n\circ e\circ b_n'\circ\cdots \circ b_1'\circ a\circ x\] is  $\alpha$-good. Notice that $j\leq h$.

Suppose for contradiction that $b'\circ a\circ x$ is not $\alpha$-good. Then neither is
$e\circ b'\circ a\circ x$. To see this, let $y=y_1\circ \cdots \circ y_r$ be a reduced form of $b'\circ a\circ x$, so that $y$ is not $\alpha$-good. Notice that a product $pq$ of two elements $p,q$ in the same vertex monoid with at least one of $p,q$ being a non-identity idempotent cannot be the identity, so that using Lemma~\ref{product reduction} iteratively we see that   $e\circ y$ is $S$-reducible. {It follows from Lemma \ref{product reduction 2} that $e\circ y$ reduces to
$$q_1\circ \cdots \circ q_m\circ y'$$ where for all $1\leq t\leq m$, $q_t=e_t$ or $q_t=e_ty_{r_t}$
 for some distinct indices
$ r_t$,   and  $y'$ is the word obtained from $y$ by deleting the letters $y_{r_t }$.} Clearly, $s(y)\subseteq s(q_1\circ \cdots \circ q_m\circ y')$, implying that $q_1\circ \cdots \circ q_m\circ y'$ is not $\alpha$-good, and hence neither is $e\circ b'\circ a\circ x$.

By assumption,
\[ z' =  b_j\circ\cdots \circ b_n\circ q_1\circ \cdots \circ q_m\circ y'\] is $\alpha$-good. We next claim that it is a reduced form. Since $s(q_t)=s(e_t)$ for $1\leq t\leq m$ and $b_j\circ\cdots \circ b_n\circ e_1\circ \cdots \circ e_m$ is a reduced form, we deduce that $b_j\circ\cdots \circ b_n\circ q_1\circ \cdots \circ q_m$ is also  reduced  by Remark~\ref{observations for Fountainicity}.   Further, it is impossible to shuffle some $b_t$ ($j\leq t\leq n$) in $z'$ and glue it to some letter in $y'$, as this would imply that in the reduced form $b_j\circ \cdots \circ b_n\circ e\circ b_n'\circ \cdots \circ b_j'$ we may shuffle  $b_t$ and glue it to $b_t'$, contradicting the fact $b\circ e\circ b'$ is  reduced.  Thus $z'$ is indeed reduced. Since $q_1\circ \cdots \circ q_m\circ y'$ is not $\alpha$-good, neither is $z'$, contradicting the fact that $[z]=[z']$ and $b_j\in N_\alpha(w)$.

We have shown that  $b'\circ a\circ x$ must  be   $\alpha$-good.  Since $v$
is not $\alpha$-good, there exists $\beta \neq \alpha$ in  the support of the reduced form of $v$  such that $(\alpha, \beta)\not \in E$. On the other hand,  $b'\circ a\circ x$
and hence $b_n'\circ\cdots \circ b_{h+1}'\circ v$ are  $\alpha$-good,
Corollary~\ref{product of reduced}  forces   there to  be some $l$ with $h< l\leq n$ such that $s(b_l')=\beta$. Since $s(b_l')=s(b_l)$ and $h\geq j$, and
$z'$  is a reduced form,
we have that $z$ is not $\alpha$-good, which again
  contradicts our initial assumption that $b_j\in N_\alpha(w)$.
\end{proof}

 We can now show that, given
$[u\circ a \circ x]=[a\circ x]$, we can replace a letter of $x$ by any corresponding element in the same $\art$-class in the relevant vertex monoid. Note that it may be we replace a letter {\em not in } $I$ by a letter {\em in} $I$. It is for this reason that our set-up for $a\circ x$ is so delicate.

\begin{lemma}\label{key-F} Let $s(x_1)=\alpha$ and let
$\tilde{x}=x_1'\circ x_2\circ\cdots\circ x_k$ where $x_1'\in M_\alpha$ is chosen so that
$x_1\,\widetilde{\mathcal{R}}\, x_1'$ in $M_\alpha$.
Then  $$[u][a\circ x]=[a\circ x]$$ implies that
$$[u][a\circ \tilde{x}]=[a\circ \tilde{x}].$$
\end{lemma}
\begin{proof}
 If $a\circ x$ is $\alpha$-good, then by Corollary \ref{case1} and Lemma~\ref{key5} (iii)
$$N_\alpha(w)=\{b_{t_1}, \cdots, b_{t_r}, e_h, b_{t_r}^{'}, \cdots, b_{t_1}^{'}, x_1\} \mbox{~or~} N_\alpha(w)=\{b_{t_1}, \cdots, b_{t_r}, b_{t_r}^{'}, \cdots, b_{t_1}^{'}, x_1\}$$ for some  $0\leq r\leq n$ and $1\leq t_1< \cdots< t_r\leq n$ and $1\leq h\leq m$.
Whether or not $a\circ x$ is $\alpha$-good,  in the case where $b_j'\notin N_\alpha(u\circ a\circ x)$ for all
$1\leq j\leq n$, we have that $b_j\not\in N_\alpha(u\circ a\circ x)$ for all $1\leq j\leq n$, by Lemma \ref{case2}, so that $N_\alpha(u\circ a\circ x)$ equals either $\{e_h, x_1\}$ or $\{x_1\}$ for some $1\leq h\leq m$.

In either of these two special cases, let $f$ be the idempotent $b_{t_1}\cdots b_{t_r}e_hb_{t_r}^{'}\cdots b_{t_1}^{'}$ or  $b_{t_1}\cdots b_{t_r}b_{t_r}^{'}\cdots b_{t_1}^{'}$;  note that we could have $f=\epsilon$.  Then by Lemma \ref{product},
$$[u][a\circ x]=[u][a\circ x_1\circ \cdots \circ x_k]=[w'][f\circ x_1], \mbox{ or }[x_1]\mbox{ if }f=\epsilon,$$
where $w'$ is the word obtained from $w$ by deleting all letters in $N_\alpha(w).$   By replacing the first letter $x_1$ of $x$ by $x_1'$ in $u\circ a\circ x$, we have $$[u][a\circ x_1'\circ x_2\circ \cdots \circ x_k]=[w'][f\circ x_1'], \mbox{ or }[x_1],$$ again by Lemma \ref{product}.

On the other hand, by applying the maps $\overline{\phi}_\alpha$ and $\overline{\psi}_\alpha$ to $[u][a\circ x]$ and $[a\circ x]$, we have $[w']=[(a\circ x)']$ and $[f\circ  x_1]=[x_1]$ (if $f\neq \epsilon$) where $(a\circ x)'$ is the word obtained from $a\circ x$ by deleting the first letter $x_1$ of $x$. The latter gives $fx_1=x_1$ in $M_\alpha$ (if $f\neq \epsilon$). If  $f\in M_\alpha$ is idempotent, then given $x_1~\widetilde{\mathcal{R}}~x_1'$ in $M_\alpha$,
we have $fx_1'=x_1'$. Therefore $$[u][a\circ x_1'\circ x_2\circ \cdots \circ x_k]=[(a\circ x)'][x_1']=[a\circ x_1'\circ x_2\circ \cdots \circ x_k]$$ so that \[[u][a\circ \tilde{x}]=[a\circ\tilde{x}].\]

 We now proceed by induction on the length of $u$. If $|u|=1$, then
$u=e_1$ for some non-identity idempotent $e_1$ from a vertex monoid. Clearly
$b_j'\notin N_\alpha(w)$ for all $j\in \{ 1,\hdots ,n\}$ so that if
$[u][a\circ x]=[a\circ x]$, then $[u][a\circ \tilde{x}]=[a\circ \tilde{x}]$,  by the above.

Suppose now that $1<|u|$ and the result is true for all idempotents having length less than $u$, when written in standard form. By the above we only need to consider the case where
 $a\circ x$ is not $\alpha$-good and there exists some $b_j'\in N_\alpha(u\circ a\circ x)$. We  pick  $j$ to be smallest such index. Then
$b_{j-1}'\circ \cdots \circ b_1'\circ a\circ x$ is $\alpha$-good.  Since $x$ is $\alpha$-good we have
$b_{j-1}'\circ \cdots\circ b_1'\circ a$ is $\alpha$-good and since $a\circ x$ is not $\alpha$-good we also have that $a$ is not $\alpha$-good. We see from Corollary~\ref{H-2} that $b_{j-1}'\circ \cdots \circ b_1'\circ a$ is not $S$-reducible. There must therefore be a smallest $t$  such that  $b_{t}'\circ \cdots \circ b_1'\circ a$ is $S$-reducible, but $b_{t+1}'\circ b_t'\circ \cdots \circ b_1'\circ a$ is not. By Lemma \ref{product reduction 2}, we know $b_{t}'\circ \cdots \circ b_1'\circ a$
shuffles to some  $$p_t\circ \cdots \circ p_1\circ a'$$  and reduces to
a reduced form $$q_t\circ \cdots \circ q_1\circ a'$$ where for all $1\leq r\leq t$ we have $q_r=b_r'=p_t$ or  $p_t=b_r'\circ a_{r_j}$ and $q_r=b_r'a_{r_j}$,
 for some distinct indices
$r_j\in \{ 1,\hdots ,l\}$, and  $a'$ is the word obtained from $a$ by deleting the letters $a_{r_j}$.

Now consider the reduced form of
$$b_{t+1}'\circ q_t\circ \cdots \circ q_1\circ a'
\mbox{ or, equivalently, } b_{t+1}'\circ p_t\circ \cdots \circ p_1\circ a'.$$  Since $s(q_r)=s(b_r')=s(p_r)$ for all $1\leq r\leq t$ and $b_{t+1}'\circ b_t'\circ \cdots \circ b_1'$ is a reduced form, we have that $b_{t+1}'\circ q_t\circ \cdots \circ q_1$ is a reduced form. As   $b_{t+1}'\circ b_t\circ \cdots \circ b_1\circ a$ and hence $b_{t+1}'\circ q_t\circ \cdots \circ q_1\circ a'$
is not $S$-reducible, there must be a letter $a_{r_{t+1}}$ in $a'$ such that $s(b'_{t+1})=s(a_{r_{t+1}})$, $b_{t+1}'a_{r_{t+1}}$ is an identity and such that
we must be able to  shuffle  $a_{r_{t+1}}$  to the front of $q_t\circ \cdots \circ q_1\circ a'$. Note that we can therefore also shuffle
$a_{r_{t+1}}$  to the front of $p_t\circ \cdots \circ p_1\circ a'$  and hence  to the front of $a$, and
$b_{t+1}'$ to the right of $p_t\circ\cdots \circ p_1$ and hence to the right of $b_t'\circ\cdots \circ b_1'$. We can therefore assume that
$t+1=1=r_{t+1}$ so that $b_1'a_1$ is an identity.

We now have
\[[u\circ a\circ x]=[b_1\circ\cdots \circ b_n\circ e\circ b_n'\circ\cdots \circ b_2'\circ a_2\cdots \circ a_l\circ x]=
[a\circ x]\]
so that multiplying by $[b_1']$ on the left we have
\begin{equation}\label{eqn5}[b_2\circ \cdots \circ b_n\circ  e\circ b_n'\circ\cdots \circ b_2'][a_2\cdots \circ a_l\circ x]
=[a_2\circ\cdots \circ a_l\circ x].\end{equation}
We note that $ a_2\cdots \circ a_l\circ x$ is a word of the correct form for us to apply our  inductive assumption, which  gives us that

\begin{equation}\label{eqn6} [b_2\circ \cdots \circ b_n\circ  e\circ b_n'\circ\cdots \circ b_2'][ a_2\cdots \circ a_l\circ \tilde{x}]=
[a_2\circ\cdots \circ a_l\circ \tilde{x}].\end{equation}
Now multiplying Equation~(\ref{eqn6}) by $[b_1]$ on the left and re-instating $b_1'\circ a_1$ we obtain
\[ [u\circ a\circ \tilde{x}]=[b_1\circ a_2\circ \cdots \circ a_l\circ \tilde{x}].\]
But multiplying Equation~(\ref{eqn5}) by $[b_1]$ on the left and re-instating $b_1'\circ a_1$ we also obtain
\[[u\circ a \circ x]=[a\circ x]=[b_1\circ a_2\circ\cdots \circ a_l\circ x].\]
 Let $x'$ be the word obtained from $x$ by deleting letters which are identities. Then \[[u\circ a \circ x']=[a\circ x']=[b_1\circ a_2\circ\cdots \circ a_l\circ x'].\]  Since $a\circ x'$ is a reduced form by Lemma \ref{key5} (ii) and $|b_1\circ a_2\circ \cdots \circ a_l\circ x'|=|a\circ x'|$,  we deduce that $b_1\circ a_2\circ \cdots\circ a_l\circ x'$ is a reduced form, so that
$[a]=[b_1\circ a_2\circ \cdots \circ a_l]$ by Lemma \ref{block}. Therefore,
\[ [u\circ a\circ \tilde{x}]= [a\circ \tilde{x}].\]
\end{proof}

\begin{Cor}\label{Rtilde} Suppose that for each $1\leq j\leq k$
we have $x_j'\in M_{s(x_j)}$ such that
$x_j\,\widetilde{\mathcal{R}}\, x_j'$ in $M_{s(x_j)}$. Let
$\bar{x}=x_1'\circ x_2'\circ\cdots\circ x_k'$. Then
\[[a\circ x]\,\widetilde{\mathcal{R}}\, [a\circ \bar{x}].\]
\end{Cor}

\begin{proof}
Suppose that $$[u][a\circ x]=[a\circ x].$$ By Lemma \ref{key-F}, we have $$[u][a\circ x_1'\circ x_2\circ \cdots \circ x_k]=[a\circ x_1'\circ x_2\circ \cdots \circ x_k].$$ Clearly,  we may shuffle $x_1'$ to the back of $x_1'\circ x_2\circ \cdots \circ x_k$ and note
that, by Lemma \ref{key5} (i), $a\circ x_2\circ \cdots \circ x_k\circ x_1'$ is of the correct form to apply Lemma~\ref{key-F}.
By repeating this process, and reshuffling, we obtain $[u][a\circ \bar{x}]
=[a\circ \bar{x}]$.

Since $a\circ \bar{x}$ is of the same form as $a\circ x$, we may show that $[u][a\circ \bar{x}]
=[a\circ \bar{x}]$ implies $[u][a\circ x]=[a\circ x]$ by exactly the same arguments as above. Therefore, $[a\circ x]\,\widetilde{\mathcal{R}}\, [a\circ \bar{x}]$.
\end{proof}

We can now prove our second main result.

\begin{Thm}\label{thm:fountain}
The graph product $\mathscr{GP}=\mathscr{GP}(\Gamma,\mathcal{M})$ of left Fountain monoids $\mathcal{M}=\{M_\alpha: \alpha\in V\}$ with respect to $\Gamma$ is  a left Fountain monoid.
\end{Thm}

\begin{proof}
Let $[w]\in \mathscr{GP}$. From Proposition~\ref{step-1} we may  write $[w]=[a][v]$, where  all letters contained in $a$ are left invertible, $a\circ v$ is a reduced form, and $v=v_1\circ \cdots \circ v_m$ is a left Foata normal form with blocks $v_i$, $1\leq i\leq m$, such that $v_1$ contains no left invertible letters; we prefer to use $v$ here since for convenience in this section we have been using $x$ to denote a single block.
Suppose that $v_1=x_1\circ\cdots \circ x_k=x$ and for each
$j\in \{ 1,\hdots ,k\}$ choose an idempotent
$x_j^+\in M_{s(x_j)}$ such that
$x_j~\widetilde{\mathcal{R}}~x_j^+$ in $ M_{s(x_j)}$. Let
$v_1^+=x_1^+\circ\cdots\circ x_k^+=\bar{x}$.
Let $[a']$ be a left inverse for $[a]$.  Using the fact that $\mathcal{R}$ and $\mathcal{R}^*$ are  left congruences contained in $\widetilde{\mathcal{R}}$, Proposition \ref{step-1} and Corollary \ref{Rtilde} give us that  $$[a][v]~\widetilde{\mathcal{R}}~[a][v_1]~\widetilde{\mathcal{R}}~[a][v_1^+]~\widetilde{\mathcal{R}}~[a][v_1^+][a'],$$
the final step following from the fact $[a']$, being right invertible,  is $\mathcal{R}$-related to the identity of $\mathscr{GP}$.
We have earlier seen that $[a][v_1^+][a']$ is an idempotent, so that
$\mathscr{GP}$ is indeed a left Fountain monoid.
\end{proof}

Of course, the left-right dual of Theorem ~\ref{thm:fountain} holds, and hence one may also deduce that the graph product of Fountain monoids is Fountain.

\section{Applications and open questions}\label{sec:application}

The aim of this section is to explore some applications of Theorems~\ref{thm:mainm} and  \ref{thm:fountain}.   Further, we will discuss some  open problems related to this work.

We make the following observation before re-obtaining one of the main results of \cite{fountain:2009}. If $M$ is a right cancellative monoid with identity 1 and $b\in M$ is a left inverse of $a\in M$, then $1a=a1=a(ba)=(ab)a$, giving  $1=ab$, so that $b$ is also a right inverse of $a$, and hence an inverse.

\begin{Cor}\cite[Theorem 1.5]{fountain:2009}\label{cor:canc}
The graph product $\mathscr{GP}=\mathscr{GP}(\Gamma,\mathcal{M})$ of right cancellative monoids $\mathcal{M}=\{M_\alpha: \alpha\in V\}$  with respect to $\Gamma$ is  right cancellative.
\end{Cor}
\begin{proof} In Proposition~\ref{prop:mainprop} we  take $z_j^+$ as the identity of the vertex monoid $M_{s(z_j)}$ for
each $1\leq j\leq s$.  By Lemma~\ref{invertible}, bearing in mind $[a]$ is a reduced form, we have that  $[a']$ as a product of left inverses (hence two-sided inverses) of the letters in $a$. Then
\[[u]\,\mathcal{R}^*\, [a][v_1^+][a']=[a][\epsilon][a']=[\epsilon],\]
and it follows from the comment after Remark~\ref{rem:a} that $\mathscr{GP}$ is right cancellative. \end{proof}

 Of course, the corresponding result is true for graph products of left cancellative, and cancellative, monoids.

\medskip

 We now turn our attention to  graph products of {\em semigroups} \cite{nouf:thesis}. This is an essentially different construction to that for monoids, since semigroups are  algebras with a different signature from that for monoids.  The combinatorics of graph products of semigroups are significantly easier to handle than graph products of monoids; they behave in a way more akin to graph monoids, where the only unit in any vertex monoid is the identity.

 As in the  case for monoids, graph products of semigroups are given by a presentation. The difference here is that a presentation denotes a quotient of a free semigroup $X^+$ on a set $X$, where
$X^+=X^*\setminus \{\epsilon\}$ is the set of non-empty words on $X$ under juxtaposition.
Still with $\Gamma=\Gamma(V,E)$, let $\mathcal{S}=\{S_\alpha: \alpha\in V\}$ be a set of semigroups, called {\it vertex semigroups}, such that $S_\beta\cap S_\gamma=\emptyset$ for all $\beta\neq \gamma\in V$.

\begin{defn}\label{defn:gps}
  The {\em graph product} $\mathscr{GPS}=\mathscr{GPS}(\Gamma,\mathcal{S})$ of  $\mathcal{S}$ with respect to $\Gamma$ is defined  by the presentation
 \[\mathscr{GPS}=\langle X\mid R^s\rangle \]
 where  $X=\bigcup_{\alpha\in V}S_\alpha$ and  $R^s=R_v\cup R_e$, with $R_v$ and $R_e$ as in Definition~\ref{defn:graphprodmonoids}.
\end{defn}

As before, identifying  a relation in $R^s$ with a pair in $X^+\times X^+$, we have \[\mathscr{GPS}=X^+/(R^s)^\sharp\] where $(R^s)^\sharp$ is the congruence on $X^+$ generated by $R^s$.

Note that, in Definition~\ref{defn:gps}, even if $S_\alpha$ and $S_\beta$ are monoids for some $\alpha,\beta\in V$, we do not identify their identities in $\mathscr{GPS}$. 
We denote the $(R^s)^\sharp$-class of $x_1\circ \cdots \circ x_n\in X^+$ in $\mathscr{GPS}$ by $\lfloor x_1\circ \cdots \circ x_n\rfloor$.
As we remarked in Section~\ref{sec:intro}, graph products of semigroups do not possess the complexities existing for monoid (or, indeed, group) graph products. Essentially, this is because (with obvious notation), for words $x,y\in X^+$ we have $s(x)\subseteq s(w)$ for any word $w$ such that $\lfloor w \rfloor= \lfloor xy \rfloor$ or $\lfloor yx \rfloor$. Moreover, if $x$ is of minimal length in its $(R^s)^\sharp$-class, then
$|x|\leq  |w|$. Details will appear in \cite{nouf:thesis}.
However, the following result will enable us to use results for graph products of monoids to deduce corresponding results for semigroups.

\begin{Prop}\label{embedding} Let $\mathscr{GPS}$ be the graph product of semigroups $\mathcal{S}=\{S_\alpha: \alpha\in V\}$ with respect to $\Gamma=\Gamma(V, E)$. For each $\alpha\in V$ let $M_\alpha$ be the semigroup
$S_\alpha$ with an identity $\underline{1}_\alpha$ adjoined {\em whether or not $S_\alpha$ is a monoid} and put
$\mathcal{M}=\{M_\alpha: \alpha\in V\}$.

Let $\mathscr{GP}$ be the graph product of monoids $\mathcal{M}$ with respect to $\Gamma$. Then the map $$\theta: \mathscr{GPS}\longrightarrow \mathscr{GP}: \lfloor x_1\circ \hdots x_n\rfloor \mapsto [ x_1\circ \hdots \circ x_n]$$ is a (semigroup) embedding.
\end{Prop}

\begin{proof} For clarity here we take $Y=\bigcup_{v\in V}S_v$ and
$X=\bigcup_{v\in V}M_v$. Let a semigroup morphism
\[\kappa:Y^+\rightarrow \mathscr{GP}\]
be defined by its action on generators  as $y\kappa=[y]$ for all $y\in Y$.
We have (with slight abuse of notation) $R^s\subseteq R$, and it follows that $\kappa$ induces the semigroup morphism
$\theta$ as given.

We now show that $\theta$ is one-one. Let $\mathscr{GPS}^1$ be the monoid obtained from $\mathscr{GPS}$ by adjoining an identity $1$. We define a monoid morphism $$\xi: X^*\longrightarrow \mathscr{GPS}^1$$ by its action on generators as
\[x\xi=\left\{ \begin{array} {ll} \lfloor x\rfloor &x\in Y\\
1&x=\underline{1}_\alpha \mbox{~for~some~} \alpha\in V.\end{array}\right.\]
We claim that $R^\sharp\subseteq \ker\xi$.

Let $u, v\in M_\alpha$ for some $\alpha\in V$. If $u, v\in S_\alpha$, then
 $$(u\circ v)\xi=(u\xi)(v\xi)=\lfloor u\rfloor \lfloor v\rfloor =\lfloor u\circ v\rfloor =\lfloor uv\rfloor =(uv)\xi.$$ If $u=\underline{1}_\alpha$ and $v\in S_\alpha$, then
 $$(u\circ v)\xi=(u\xi)(v\xi)=1\lfloor v\rfloor =\lfloor v\rfloor =v\xi=(uv)\xi$$
and dually if $u\in S_\alpha$ and $v=\underline{1}_\alpha$. If $u=v=\underline{1}_\alpha$, then \[(u\circ v)\xi=(u\xi)(v\xi)=11=1=(uv)\xi.\]

Now consider $u\in M_\alpha, v\in M_\beta$ with $(\alpha, \beta)\in E$. If $u=\underline{1}_\alpha$ and $v=\underline{1}_\beta$, then
 $$(u\circ v)\xi=(\underline{1}_\alpha\circ \underline{1}_\beta)\xi=(\underline{1}_\alpha\xi)(\underline{1}_\beta\xi)=11=(\underline{1}_\beta\xi)(\underline{1}_\alpha\xi)=(\underline{1}_\beta\circ \underline{1}_\alpha)\xi=(v\circ u)\xi.$$ If $u=\underline{1}_\alpha$ and $v\in S_\beta$, then $$(u\circ v)\xi=1 \lfloor v\rfloor =\lfloor v\rfloor  1=(v\circ u)\xi$$ and dually if $u\in S_\alpha$ and $v=\underline{1}_\beta$. If $u\in S_\alpha$ and $v\in S_\beta$, then $$(u\circ v)\xi=\lfloor u\rfloor \lfloor v\rfloor =\lfloor u\circ v\rfloor =\lfloor v\circ u\rfloor =\lfloor v\rfloor  \lfloor u\rfloor =\lfloor v\circ u\rfloor \xi.$$

Finally, for $\alpha\in V$, we have $\underline{1}_\alpha\xi=1=\epsilon \xi$.

We have shown that  $R\subseteq \ker\xi$. It follows that
$R^\sharp\subseteq \ker\xi$ and hence
 $$\overline{\xi}: \mathscr{GP}\longrightarrow \mathscr{GPS}^1, \,\, [w]\mapsto w\xi$$ is a well defined morphism.
Further, for any $\lfloor w\rfloor \in \mathscr{GPS}$, we have $$\lfloor w\rfloor \theta\overline{\xi}=[w]\overline{\xi}=w\xi=\lfloor w\rfloor$$ so that $\theta\overline{\xi}=1_{\mathscr{GPS}}$, and hence $\theta$ is  an embedding.
\end{proof}

The result below will appear in \cite{nouf:thesis}.

\begin{Cor}\label{cor:semigroups}
The graph product $\mathscr{GPS}$ of left abundant semigroups $\mathcal{S}=\{S_\alpha: \alpha\in V\}$ with respect to $\Gamma$ is left abundant.
\end{Cor}

\begin{proof}
Let $Y=\bigcup_{\alpha\in V}S_\alpha$ and $X=\bigcup_{\alpha\in V}M_\alpha$, where $M_\alpha=S_\alpha\cup \{\underline{1}_\alpha\}$ as in Proposition \ref{embedding}.
 Since each $S_\alpha$ is left abundant, it is easy to check that the same is true of each $M_\alpha$,
and, moreover, if $u,v\in S_\alpha$ then $u\,\mathcal{R}^*\, v$ in $S_\alpha$ if and only if
$u\,\mathcal{R}^*\, v$  in $M_\alpha$.

 It follows from Proposition \ref{embedding} that $\mathscr{GPS}$ is isomorphic to a subsemigroup $\mathscr{N}$ of $\mathscr{GP}$, where $$\mathscr{N}=\{[x_1\circ \cdots \circ x_n]: x_i \in Y, 1\leq i\leq n\}$$ and $$\varphi: \mathscr{GPS}\longrightarrow \mathscr{N}, \lfloor x_1\circ \cdots \circ x_n\rfloor\mapsto [x_1\circ \cdots \circ x_n]$$ is an isomorphism.

 Let $x=x_1\circ \cdots \circ x_n\in Y^+$ and let $v=v_1\circ \cdots \circ v_m\in X^*$ be a left Foata normal form of $x$ with blocks  $v_i$,   $1\leq i\leq m$. Since the only left or right invertible element of
 any vertex monoid $M_\alpha$ is $\underline{1}_\alpha$, we have that
 $v\in Y^+$  and  $v$ contains no left invertible letters.
 Choosing $v_1^+\in Y^+$ as in Proposition~\ref{prop:mainprop} and noticing that
 $a=\epsilon$ in that result, we have that
 $[x]=[v]\,\mathcal{R}^*\, [v_1^+]$ in  $\mathscr{GP}$ and  hence in $\mathscr{N}$. It follows that
 $\lfloor x\rfloor \,\mathcal{R}^*\, \lfloor v_1^+\rfloor$ in $\mathscr{GPS}$.
\end{proof}

The proof of the following result is similar to that of Corollary \ref{cor:semigroups}.
\begin{Cor}\label{fountain-cor:semigroups}
The graph product $\mathscr{GPS}$ of left Fountain semigroups $\mathcal{S}=\{S_\alpha: \alpha\in V\}$ is a left Fountain semigroup.
\end{Cor}

Of course, the right (two-sided) versions of Corollaries \ref{cor:semigroups} and ~\ref{fountain-cor:semigroups} also hold.

We  remarked in Section~\ref{sec:preliminaries} that free products and restricted direct products of  monoids can be regarded as special cases of graph products of monoids. We therefore have the following result.

\begin{Cor}\label{cor:fprdpm} The
free product $\mathscr{FPM}$ and the restricted direct product $\oplus_{\alpha\in V}M_\alpha$ of left abundant monoids (resp. left Fountain monoids) $\mathcal{M}=\{M_\alpha: \alpha\in V\}$ are left abundant (resp. left Fountain).
\end{Cor}

 \begin{Rem} The corresponding statement to that of Corollary~\ref{cor:fprdpm} is true for semigroups and in the right/two-sided
 case for both monoids and semigroups.
\end{Rem}

We finish this paper by posing the following open problems. Let $M$ be a monoid. We have commented that the relations $\ar$ and $\ars$ are left  congruences on $M$ but, in general,  this need not be true of $\art$.  Since $\art$ being a left congruence is an important property in many structural results for left Fountain monoids and semigroups we first pose:

\begin{Question} Let
 $\mathscr{GP}=\mathscr{GP}(\Gamma,\mathcal{M})$ be a graph product of   monoids $\mathcal{M}=\{M_\alpha: \alpha\in V\}$ with respect to $\Gamma$, where $\art$ is a left congruence on each $M_{\alpha}$. Is $\art$ a left congruence on  $\mathscr{GP}$?
\end{Question}

The above could first be asked in the corresponding case for semigroups, and starting with the vertex semigroup being left Fountain.

A monoid is {\em inverse} if it is regular and its idempotents commute. Inverse monoids form a variety not of monoids but of {\em unary monoids},
that is, monoids equipped with an additional unary operation. In this case the unary operation is given by $a\mapsto a^{-1}$, where $a^{-1}$ is the unique element such that $a=aa^{-1}a$ and $a^{-1}=a^{-1}aa^{-1}$. The notion of a graph product of {\em inverse monoids} (see \cite{daCosta:2003,diekert:2008}, at least for the case where the vertex monoids are free) is analogous to that for monoids and semigroups, and is obtained as a quotient of a free inverse monoid, by relations given as for $R$; from its very construction, it is inverse. A  monoid is {\em left adequate} if it is left abundant and its idempotents commute. These are the first non-regular analogues of inverse monoids, and  form quasivarieties of unary monoids.  Here the unary operation is $a\mapsto a^+$ where $a^+$ is the {\em unique} idempotent in the $\ars$-class of $a$. We therefore ask the following question, which can be interpreted in more than one way. Of course, one could also begin with the semigroup case.

\begin{Question} Is the graph product of left adequate monoids left adequate?
\end{Question}

Finally, we would hope that using left Foata normal forms and other reduction techniques developed in this article we could
both find new approaches to old results (such as calculating centralizers in graph products of groups \cite{barkauskas:2007}) and  extend these to the monoid case. For example, we ask:

\begin{Question} Determine centralisers in graph products of  monoids.
\end{Question}

\section*{Acknowledgement} The authors are grateful to a careful referee for providing some very useful remarks which have improved the readability of the final version.

\end{document}